\documentclass[11pt,a4paper]{article}

\usepackage{amsfonts,amsmath,amssymb,amsthm}
\usepackage{graphicx} 
\usepackage{subfig}
\usepackage{algorithmic}
\usepackage{algorithm}
\usepackage{url}
\usepackage{xcolor}
\usepackage{multirow}
\usepackage[titletoc,title]{appendix}
\usepackage{threeparttable} 
\usepackage{authblk}

\newtheorem{theorem}{Theorem}
\newtheorem{lemma}{Lemma}
\newtheorem{remark}{Remark}
\newtheorem{assumption}{Assumption}
\newtheorem{definition}{Definition}
\newcommand{\cG}{\mathcal{G}}

\newcommand{\ve}{\varepsilon}
\newcommand{\argmin}{\arg\min}
\newcommand\keywords[1]{\textbf{Keywords}: #1}

\title{An Inexact Proximal Newton Method for Nonconvex Composite Minimization} 
\author{Hong~Zhu\thanks{zhuhongmath@126.com}}
\affil{School of Mathematical Sciences, Jiangsu University, Zhenjiang, 212013, Jiangsu, China.}
\date{}

\begin{document}

\maketitle

\begin{abstract}
In this paper, we propose an inexact proximal Newton-type method for nonconvex composite problems. We establish the global convergence rate of the order \(\mathcal{O}(k^{-1/2})\) in terms of the minimal norm of the KKT residual mapping and the local superlinear convergence rate in terms of the sequence generated by the proposed algorithm under the higher-order metric \(q\)-subregularity property. When the Lipschitz constant of the corresponding gradient is known, we show that the proposed algorithm is well-defined without line search. 
Extensive numerical experiments on {the \(\ell_1\)-regularized Student's \(t\)-regression and the group penalized Student's \(t\)-regression} show that the performance of the proposed method is comparable to the state-of-the-art proximal Newton-type methods.

.
%
\end{abstract}

\section{Introduction}\label{intro}
In this paper, we consider the problem of the form
\begin{equation}\label{eq:ncp}
\min_x\varphi(x) : = f(x) + g(x),
\end{equation} 
where \(f: \mathbb{R}^n\to(-\infty, +\infty]\) is twice continuously differentiable, \(g: \mathbb{R}^n\to (-\infty, +\infty]\) is proper, convex, and lower semicontinuous. We assume that the solution set of Problem~\eqref{eq:ncp} is non-empty and denote \(\varphi_*\) as the optimal function value. 

The proximal gradient method (PGM) is one of the efficient algorithms for solving Problem~\eqref{eq:ncp}. The general update step of PGM for solving Problem~\eqref{eq:ncp} takes the form 
\[
x_{k+1} = \argmin_x\{l_k(x; x_k) + \frac{1}{2t_k}\|x - x_k\|^2\},
\] 
where \(l_k(x; x_k) =  f(x_k) + \langle\nabla f(x_k), x - x_k\rangle + g(x)\) and \(t_k > 0\) refers to the step size at iteration \(k\). Define \(\mathcal{G}(x) = x - {\rm prox}_g(x - \nabla f(x))\) as the KKT residual mapping of Problem~\eqref{eq:ncp}, where \({\rm prox}_g(u) := \argmin_x\{g(x) + \frac{1}{2}\|x - u\|^2\}\). We assume that \({\rm prox}_g(\cdot)\) is efficiently solvable.  When \(f\) is convex, the convergence rates in terms of the function values and the minimal norm of the KKT residual mapping generated by PGM  are \(\mathcal{O}(1/k)\)~\cite[Sec. 10.4]{B17}. When \(f\) is nonconvex, the convergence rate in terms of the minimal norm of the KKT residual mapping generated by PGM is \(\mathcal{O}(1/\sqrt{k})\)~\cite[Sec. 10.3]{B17}.

When \(f\) is twice continuously differentiable, it is natural to expect that the algorithm can have a faster convergence rate after invoking the Hessian \(\nabla^2f(x_k)\) of \(f\) at each iteration, that is, for given \(k\in\mathbb{N}\), considering the following problem
\begin{equation}\label{eq:smajf}
\min_x\{q(x; x_k, H_k): = l_k(x; x_k) + \frac{1}{2}(x - x_k)^\top H_k(x - x_k)\},
\end{equation}
where \(H_k\) is a symmetric positive definite matrix. Methods that solving Problem~(\ref{eq:ncp}) via minimizing Problem~\eqref{eq:smajf} are named as the proximal Newton method (also named as the sequential quadratic approximation (SQA) method) if \(H_k = \nabla^2f(x_k) + \mu_k I\) for some \(\mu_k > 0\) and named as the proximal quasi-Newton method if \(H_k = B_k + \mu_k I\), where the symmetric positive semidefinite matrix \(B_k\) is an approximation of \(\nabla^2f(x_k)\). Let \(\tilde{x}_k\) be a (or an inexact) solution of Problem~\eqref{eq:smajf}. Then \(x_{k+1} = x_k + t_k(\tilde{x}_k - x_k)\) and the step length \(t_k\) is obtained by performing a backtracking line search along the direction \(\tilde{x}_k - x_k\). Different accuracy criteria on \(\tilde{x}_k\) and line search conditions are used in literature to ensure convergence. 
For problems with convex \(f\),~authors of \cite{LSS12,BNO16} demonstrated the superlinear local convergence rate of the proximal Newton method when \(g(x) = \lambda\|x\|_1\) for some \(\lambda > 0\). 
Lee et. al~\cite{LSS14} demonstrated the superlinear local convergence rate of an inexact proximal Newton method for a generic \(g\). However, the global convergence was not provided. 
Various globalizations of the (inexact) proximal Newton-type methods for Problem~\eqref{eq:ncp} with convex \(f\) can be found in the literature; see, e.g.,~\cite{ST16,LW19,YZS19,MYZZ22}. 
Scheinberg and Tang~\cite{ST16} provided the sublinear global convergence rate of an inexact proximal Newton-type method in terms of function values. 
Authors of \cite{YZS19} and \cite{MYZZ22} used \(\|\mathcal{G}(x_k)\|^{\delta}\) as the regularization parameter in \(H_k\) and analyzed the local convergence rates under the Luo-Tseng error bound condition and the metric \(q\)-subregularity property, respectively.   

For the problem with nonconvex \(f\), Lee and Wright~\cite{LW19} proposed an inexact SQA method with backtracking line search. The convergence rate in terms of the minimal norm of \(\mathcal{G}(x_k)\) is \(\mathcal{O}(1/\sqrt{k})\). However, the local convergence was not provided. 
Kanzow and Lechner~\cite{KL21} proposed a globalized inexact proximal Newton-type method by switching from a Newton step to a proximal gradient step when the proximal Newton direction does not satisfy a sufficient decrease condition. The superlinear convergence rate was established under the local strong convexity of \(\varphi\). Lee~\cite{L22} proposed an accelerating inexact SQA method for partly smooth regularizers. The running time can be improved by identifying the active manifold that makes the objective function smooth. The convergence rate in terms of minimal norm of \(\mathcal{G}(x_k)\) is \(\mathcal{O}(1/\sqrt{k})\). The two-step superlinear convergence rate of the iterate sequence was given under a sharpness condition on \(\varphi_{\phi}(x)\) around \(y^*\), where \(y^*\) satisfies \(\phi(y^*) = x^*\) for some stationary point \(x^*\) of Problem~\eqref{eq:ncp} and \(\phi\) is a parameterization of the active manifold such that \(\varphi_{\phi}(x)\) is convex around \(y^*\). Recently, Liu et al.~\cite{LPWY23} proposed an inexact regularized proximal Newton method and demonstrated the superlinear convergence rate of the iterate sequence provided that cluster points satisfy a H\"{o}lderian error bound or the higher-order metric subregularity property.  
As we can see, much of the existing literature focuses on the global or local convergence rate. In this paper, we propose an inexact proximal Newton method for Problem~\eqref{eq:ncp} with nonconvex \(f\) and provide analyses for both global and local convergence rates. 
We assume the following standard assumptions hold. 
\begin{assumption}\label{assume:ncp}
\begin{itemize}
\item[(i)] \(f: \mathbb{R}^n\to(-\infty, +\infty]\) is twice continuously differentiable on an open set \(\Omega_1\) containing the effective domain \({\rm dom}(g)\) of \(g\) and \(\nabla f\) is \(L_H\)-Lipschitz continuous over \(\Omega_1\).
\item[(ii)] \(g: \mathbb{R}^n\to(-\infty, +\infty]\) is proper closed convex, nonsmooth and lower semicontinuous.
\item[(iii)] For any \(x_0\in {\rm dom}(g)\), the level set \(\mathcal{L}_{\varphi}(x_0) = \{x\vert \varphi(x) \leq \varphi(x_0)\}\) is bounded. 
\end{itemize}
\end{assumption}

Our main contributions can be summarized as follows:
\begin{enumerate}
\item[(a)] We propose an inexact proximal Newton method with line search for Problem~\eqref{eq:ncp} and show that every accumulation point of the sequence generated by our proposed method is a stationary point. We establish that the global convergence rate of our proposed method is \(\mathcal{O}(1/\sqrt{k})\) in terms of the minimal norm of \(\cG(x_k)\). We provide the superlinear convergence rate of the sequence generated by the algorithm under the higher-order metric \(q\)-subregularity assumption on \(\mathcal{G}(x)\). 
\item[(b)] When the Lipschitz constant \(L_H\) is known, we show that the inexact proximal Newton method without line search is well-defined. Moreover, the sublinear global convergence rate and the superlinear local convergence rate still hold in this case under the higher-order metric \(q\)-subregularity property. 
\item[(c)] If \(g(x)\equiv 0\), then our proposed method reduces to the regularized Newton method with the global complexity bound \(\mathcal{O}(\epsilon^{-2})\). When the Lipschitz constant \(L_H\) is known, the reduced regularized Newton method takes unit stepsize with the global complexity bound \(\mathcal{O}(\epsilon^{-2-\delta})\), where \(\delta\in[0, 1]\) is a parameter that used to define \(H_k\).
\end{enumerate}

The rest of the paper is organized as follows. In Section~\ref{sec:pnmethods}, we present our inexact proximal Newton method and establish its global and local convergence rates. In Section~\ref{appendix: lhisgiven}, we discuss a special case where \(L_H\) is known and used to define \(H_k\). Numerical comparisons with the state-of-the-art proximal Newton-type methods on the \(\ell_1\)-regularized Student's \(t\)-regression and the group penalized Student's \(t\)-regression are given in Section~\ref{sec:numerical}. We make some conclusions in Section~\ref{sec:con}.

\section{Inexact proximal Newton method}\label{sec:pnmethods}
In this section, we present the inexact proximal Newton method for Problem~\eqref{eq:ncp} and give its global and local convergence rate.  

\subsection{The inexact proximal Newton method with backtracking line search}\label{sec:pmm}
Given the current iterate \(x_k\), let the semidefinite matrix \(B_k\) be an approximation of \(\nabla^2f(x_k)\) satisfying the standard boundedness assumption:
\begin{equation}\label{eq:assumbk}
\exists~M\geq 0\quad {\rm s.t.}~\|B_k\|\leq M, \quad \forall k \in\mathbb{N}. 
\end{equation}
Condition \eqref{eq:assumbk} holds for \(B_k = \nabla^2f(x_k) + \max\{0, -\lambda_{\min}(\nabla^2f(x_k))\}I\) with \(M = 2L_H\) under Assumption~\ref{assume:ncp} (i). 
Denote
\[
H_k = B_k + (c + \mu_1\min\{1, \|\mathcal{G}(x_k)\|^{\delta}\})I
\]
for some \(c>0\), \(\mu_1\in(0, 1]\), and \(\delta\in[0, 1]\). \(H_k\) is uniformly positive definite under above definition\footnote{When \(f\) is convex, we can set \(H_k = \nabla^2f(x_k) + (c + \mu_1\min\{1, \|\mathcal{G}(x_k)\|^{\delta}\})I\).}. Considering the inexact proximal Newton method given in Algorithm~\ref{alg:pnewton}. 

\begin{algorithm*}[th!]
\caption{Inexact Proximal Newton method with backtracking line search.}\label{alg:pnewton}
\begin{algorithmic}[1]
\REQUIRE{\(x_0\in {\rm dom}(g)\), \(\epsilon_0 > 0\), \(c > 0\), \(\tau\in(0, c)\), \(\mu_1\in(0, 1]\), \(\mu_2\in(0, \mu_1]\), \(\delta \in [0, 1]\), and \(\theta\in(0, 1)\). }
\FOR{\(k = 0, 1, \cdots, \)}
\STATE{compute \(\|\cG(x_k)\|\) and \(H_k\);}
\STATE{if \(\|\cG(x_k)\| \leq \epsilon_0\), then terminate the algorithm;}
\STATE{compute}
\begin{equation}\label{eq:xsub}
\hat{x}_k = \argmin_{x}\{q(x; x_k, H_k) + \ve_k^\top(x - x_k)\},
\end{equation}
where the error vector \(\ve_k\) satisfies the following accuracy condition
\begin{equation}\label{eq:errvek}
\|\ve_k\| \leq \frac{\mu_2}{2}\min\{1, \|\cG(x_k)\|^{\delta}\}\|\hat{x}_k - x_k\|. 
\end{equation}
\STATE{set \(d_k = \hat{x}_k - x_k\), \(x_{k+1} = x_k + \alpha_kd_k\), where \(\alpha_k = \theta^{j_k}\) and \(j_k\) is the smallest nonnegative integer such that  }
\begin{equation}\label{eq:ls}
\varphi(x_k + \theta^{j_k}d_k) \leq \varphi(x_k) - \frac{\tau}{2}\theta^{j_k}\|d_k\|^2
\end{equation}
\ENDFOR
\RETURN{\(\{x_k\}\)}
\end{algorithmic}
\end{algorithm*}

In Algorithm~\ref{alg:pnewton}, equation~\eqref{eq:xsub} is a representation of inexactly solving Problem~\eqref{eq:smajf}. Inequality~\eqref{eq:errvek} provides the accuracy criteria on \(\hat{x}_k\) and inequality~\eqref{eq:ls} gives the line search condition on the step size \(\alpha_k\).
Let \(x_k^*\) be the exact solution of Problem~\eqref{eq:smajf}. Then \(\hat{x}_k = x_k^*\) when \(\varepsilon_k = 0\) in equation~\eqref{eq:xsub}. Next, we show equations~\eqref{eq:xsub} and~\eqref{eq:errvek} are well defined. According to the first-order optimality condition of Problem~\eqref{eq:smajf}, we have \(0 \in \partial g(x_k^*) + \nabla f(x_k) + H_k(x_k^* - x_k)\). For equation~\eqref{eq:xsub}, by using the first-order optimality condition and the definition of \(q(x; x_k, H_k)\), we have 
\[
0 \in \partial g(\hat{x}_k) + \nabla f(x_k) + H_k(\hat{x}_k - x_k) + \varepsilon_k.
\]
Hence, \(\hat{x}_k\) can be seen as an inexact solution of Problem~\eqref{eq:smajf} when \(\varepsilon_k \neq 0\) and \eqref{eq:xsub} and \eqref{eq:errvek} can be equivalently stated as:
\[
\hat{x}_k \approx \arg\min_x\{q(x; x_k, H_k)\}
\]
such that there exists \(\varepsilon_k\in \partial q(\hat{x}_k; x_k, H_k) = \partial g(\hat{x}_k) + \nabla f(x_k) + H_k(\hat{x}_k - x_k)\) satisfying inequality~\eqref{eq:errvek}.  
It is not hard to obtain an element that belongs to \(\partial g(y)\) by using the first-order optimal condition at a point \(y\) when the proximal-type method is used to minimize \(q(x; x_k, H_k)\). For any algorithm used to solve \(\min_xq(x; x_k, H_k)\)  satisfies \(x_{k,j}\to x_k^*\) as \(j \to \infty\), where \(\{x_{k,j}\}_{j\in\mathbb{N}}\) is the sequence generated by the algorithm (noting that \(q(x; x_k, H_k)\) is strongly convex with respect to \(x\), strong convergence on \(\{x_{k,j}\}_{j\in\mathbb{N}}\) is not demanding). If \(\{x_{k,j}\}_{j\in\mathbb{N}}\) further satisfies \(J_{k,j}(x_{k,j} - x_{k, j+1})\in \partial q(x_{k, j+1}; x_k, H_k)\), \(\forall j\in\mathbb{N}\), for some uniformly bounded matrices \(\{J_{k,j}\}_{j\in\mathbb{N}}\), then it can be proved by contradiction that 
\[
\|J_{k,j}(x_{k,j} - x_{k, j+1})\| > \frac{\mu_2}{2}\min\{1, \|\mathcal{G}(x_k)\|^{\delta}\}\|x_{k, j+1} - x_k\|, \quad \forall j\in\mathbb{N}
\]
will never hold for any \(x_k\) satisfies \(\|\mathcal{G}(x_k)\| > \epsilon_0 > 0\). Hence, we can set \(\hat{x}_k := x_{k, j+1}\), where \(x_{k, j+1}\) satisfies 
\[
\|J_{k,j}(x_{k,j} - x_{k, j+1})\| \leq \frac{\mu_2}{2}\min\{1, \|\mathcal{G}(x_k)\|^{\delta}\}\|x_{k, j+1} - x_k\|. 
\]
It can be verified that both the proximal gradient method~\cite{B17} and the FIAST method~\cite{B17} can be used to minimize \(q(x; x_k, H_k)\).

\begin{remark}\label{remark:subac}
Denote \(q_k(x) \triangleq q(x; x_k, H_k)\) and \(r_k(x) = x - {\rm prox}_g(x - (\nabla f(x_k) + H_k(x - x_k)))\). Let \({\rm dist}(x, S)\) be the Euclidean distance of a vector \(x\) to the set \(S\).  
In the following, we name Problem~\eqref{eq:smajf} as the \(x\)-subproblem of the proximal Newton method. 
Table~\ref{tab:ac} lists the accuracy criteria for \(x\)-subproblem used in the literature. Compared to other existing methods, inequality~\eqref{eq:errvek} is easy to check and the order between \(q_k(\hat{x}_k)\) and \(\varphi(x_k)\) is not explicitly required as an accuracy criterion in Algorithm~\ref{alg:pnewton}.  

\begin{table}[h!]
\setlength{\abovecaptionskip}{0cm}
\renewcommand{\arraystretch}{1.3}
\scalebox{0.87}{
\begin{threeparttable}
\caption{Accuracy criteria for \(x\)-subproblem used in literature.}\label{tab:ac}
\begin{tabular}{|c|c|l|}\hline
 Ref. & \(f\) &  ~~~~~~~~~~~~~~~~~~~~~~~~~~~~~~~~Accuracy criteria \\ \hline
 \multirow{2}{*}{\cite{ST16}}   & \multirow{2}{*}{convex}  &   \(q_k(\tilde{x}_k) \leq  \varphi(x_k)\)  \\ 
 & &\(q_k(\tilde{x}_k)\leq \min_x\{q_k(x)\} + \epsilon_k\), \quad \(\epsilon_k > 0\) \\\hline 
 \multirow{2}{*}{\cite{YZS19}}  & \multirow{2}{*}{convex} & \(\|r_k(\tilde{x}_k)\| \leq \eta\min\{\|\mathcal{G}(x_k)\|, \|\mathcal{G}(x_k)\|^{1 + \rho}\}\),~~\(\eta\in(0,1)\),~\(\rho \in[0,1]\)\\
 & &\(q_k(\tilde{x}_k) - \varphi(x_k) \leq \zeta(l_k(\tilde{x}_k; x_k) - \varphi(x_k))\),~~\(\zeta \in (\theta, 1/2)\),~\(\theta \in (0, 1/2)\)\\ \hline
  \multirow{2}{*}{\cite{MYZZ22}} & \multirow{2}{*}{convex} & \(\|r_k(\tilde{x}_k)\| \leq \nu\min\{\|\mathcal{G}(x_k)\|, \|\mathcal{G}(x_k)\|^{1+\rho}\}\), \quad \(\nu \in[0,1)\), \(\rho > 0\)\\
  & &\(q_k(\tilde{x}_k) \leq \varphi(x_k)\) \\ \hline
  \cite{LW19} & nonconvex & \(q_k(\tilde{x}_k) \leq \varphi(x_k) + (1 - \eta)\{\min_x\{q_k(x)\} - \varphi(x_k)\}\), \quad \(\eta > 0\) \\ \hline
  \multirow{3}{*}{\cite{KL21}} & \multirow{3}{*}{nonconvex} & \(\|r_k(\tilde{x}_k)\| \leq \eta_k\|\mathcal{G}(x_k)\|\), \quad \(\eta_k\in[0, \eta)\),~\(\eta \in(0, 1)\)\\
  &&\(q_k(\tilde{x}_k) - \varphi(x_k) \leq \xi\Delta_k\)~{\rm and}~\( \xi\Delta_k \leq -\rho\|\tilde{x}_k - x_k\|^p\), \quad \(\rho > 0\)\\ 
  &&~~~~~~~~~~~~~~~~~~~~~~~~~~~~~~~~~~~\(p > 2\),~\(\sigma\in(0, 1/2)\), ~\(\xi \in (\sigma, 1/2)\)\\ \hline
 \multirow{3}{*}{\cite{LPWY23}} & \multirow{3}{*}{nonconvex\tnote{*}} & \(\|r_k(\tilde{x}_k)\| \leq \eta\min\{\|r_k(x_k)\|, \|r_k(x_k)\|^{1+\tau}\}\)~if \(\rho\in(0, 1)\); \\ 
 && \({\rm dist}(0, \partial q_k(\tilde{x}_k) \leq \eta\|r_k(x_k)\|\)~if~\(\rho =0\)\\
 && and \(q_k(\tilde{x}_k) \leq \varphi(x_k)\),~\(\eta\in(0, 1)\),~\(\tau \geq \rho\)\\ \hline
   \multirow{2}{*}{Alg.\ref{alg:pnewton}} & \multirow{2}{*}{nonconvex} & \({\rm dist}(0, \partial q_k(\hat{x}_k)) \leq \frac{\mu_2}{2}\min\{1, \|\cG(x_k)\|^{\delta}\}\|\hat{x}_k - x_k\|\),\quad \(\delta\in[0,1]\)\\
   &&~~~~~~~~~~~~~~~~~~~~~~~~~~~~~~~~~~~~~~~~~~ \(\mu_2\in(0, \mu_1]\) with \(\mu_1\in(0, 1]\) \\ \hline
\end{tabular}
\begin{tablenotes}
       \item[*] In [20], \(H_k = \nabla^2f(x_k) + a_1\max\{-\lambda_{\min}(\nabla^2f(x_k)), 0\} + a_2\|\mathcal{G}(x_k)\|^{\rho}\) for some \(a_1\), \(a_2 > 0\)~and~\(\rho\in[0, 1)\). 
     \end{tablenotes}
\end{threeparttable}}
\end{table}
\end{remark}

\subsection{Global convergence }

In this section, we give the global convergence analysis of Algorithm~\ref{alg:pnewton} with respect to the minimal norm of \(\mathcal{G}(x_k)\). 
Denote \(\mathcal{S}^*\) as the set of stationary points of Problem~\eqref{eq:ncp}. The following property holds. 
\begin{lemma}\label{lem:sp}
\(\bar{x}\in \mathcal{S}^*\) if and only if \(\mathcal{G}(\bar{x}) = 0\). 
\end{lemma}
\begin{proof}
By the second prox theorem~\cite[Th. 6.39]{B17}, \(\mathcal{G}(\bar{x}) = 0\) if and only if \((\bar{x} - \nabla f(\bar{x})) - \bar{x} \in \partial g(\bar{x})\), which is exactly \(0 \in \nabla f(\bar{x}) + \partial g(\bar{x}) = \nabla\varphi(\bar{x})\).  
\end{proof}
The proof of Lemma~\ref{lem:sp} explains why we name \(\mathcal{G}(x)\) as the KKT residual mapping of Problem~\eqref{eq:ncp}. 

\begin{lemma}\label{lemma:dqkng}
Let \(\{x_k\}\) be the sequence generated by Algorithm~\ref{alg:pnewton}. We have 
\[
q_k(\hat{x}_k) \leq \varphi(x_k) + \frac{\mu_2}{2}\min\{1, \|\mathcal{G}(x_k)\|^{\delta}\}\|d_k\|^2. 
\]
\end{lemma}
\begin{proof}
By the optimality condition of \(x\)-subproblem in Algorithm~\ref{alg:pnewton}, we have \(0\in \partial q_k(\hat{x}_k) + \ve_k\). Combining with the convexity of \(q_k\) and the accuracy criteria~\eqref{eq:errvek} on \(\ve_k\), we have 
\begin{align*}
q_k(x_k) \geq& q_k(\hat{x}_k) - \ve_k^\top (x_k - \hat{x}_k) \geq q_k(\hat{x}_k) - \|\ve_k\|\|x_k - \hat{x}_k\|\\
\geq & q_k(\hat{x}_k) - \frac{\mu_2}{2}\min\{1, \|\cG(x_k)\|^{\delta}\}\|\hat{x}_k - x_k\|^2.
\end{align*}
Combing the above inequality with \(\varphi(x_k) = q_k(x_k)\), the statement holds. 
\end{proof}

\begin{lemma}\label{lem:ngk}
Suppose Assumption~\ref{assume:ncp} holds. Let \(\{x_k\}\) be the sequence generated by Algorithm~\ref{alg:pnewton}.  Denote \(\hat{\eta} = 3 + M + c\). Then
\[
\|\cG(x_k)\| \leq \hat{\eta}\|d_k\|. 
\]
\end{lemma}
\begin{proof}
Recall the definition of \(r_k(x)\) given in Remark~\ref{remark:subac}. On one hand, we have 
\begin{equation}\label{eq:ek1}
\hat{x}_k - r_k(\hat{x}_k) = {\rm prox}_g(\hat{x}_k - \nabla f(x_k) - H_k(\hat{x}_k - x_k)). 
\end{equation}
On the other hand, by the definition of \(\hat{x}_k\) and \(q_k\), we have 
\begin{equation}\label{eq:delta1}
\hat{x}_k ={\rm prox}_g(\hat{x}_k - \nabla f(x_k) - H_k(\hat{x}_k - x_k) - \ve_k). 
\end{equation}
Using the nonexpansivity of \({\rm prox}_g\)~\cite[Th. 6.42]{B17},~\eqref{eq:ek1} and~\eqref{eq:delta1} yield 
\begin{equation}\label{eq:evare}
\|r_k(\hat{x}_k)\| \leq \|\ve_k\|.
\end{equation}
Notice that \eqref{eq:ek1} also implies
\begin{equation}\label{eq:ek2}
r_k(\hat{x}_k) - \nabla f(x_k) - H_k(\hat{x}_k - x_k) \in \partial g(\hat{x}_k - r_k(\hat{x}_k)). 
\end{equation}
From the definition of \(\mathcal{G}(x)\), we have 
\begin{equation}\label{eq:cg}
\cG(x_k) - \nabla f(x_k) \in \partial g(x_k - \cG(x_k)).
\end{equation}
Using the monotonicity of \(\partial g\),~\eqref{eq:ek2} and~\eqref{eq:cg} yield 
\[
\langle \cG(x_k) + H_k(\hat{x}_k - x_k) - r_k(\hat{x}_k), x_k - \cG(x_k) - \hat{x}_k + r_k(\hat{x}_k)\rangle \geq 0,
\]
which leads to
\begin{align*}
\|\mathcal{G}(x_k) - r_k(\hat{x}_k)\|^2 \leq& \langle \mathcal{G}(x_k) - r_k(\hat{x}_k), x_k - \hat{x}_k\rangle + \langle H_k(\hat{x}_k - x_k), x_k - \hat{x}_k\rangle \\
&- \langle H_k(\hat{x}_k - x_k), \mathcal{G}(x_k) - r_k(\hat{x}_k)\rangle \\
\leq& \langle \mathcal{G}(x_k) - r_k(\hat{x}_k), x_k - \hat{x}_k + H_k(x_k - \hat{x}_k)\rangle,
\end{align*}
where the second inequality holds since \(H_k\) is positive definite. Then by Cauchy inequality, we have 
\(\|\cG(x_k) - r_k(\hat{x}_k)\| \leq \|x_k - \hat{x}_k + H_k(x_k - \hat{x}_k)\|\). 
Therefore, 
\[
\|\cG(x_k)\| \leq \|\cG(x_k) - r_k(\hat{x}_k)\| + \|r_k(\hat{x}_k)\| \leq   \|x_k - \hat{x}_k + H_k(x_k - \hat{x}_k)\| + \|r_k(\hat{x}_k)\|.
\]
Notice that \(H_k = B_k + (c + \mu_1\min\{1, \|\mathcal{G}(x_k)\|^{\delta}\})I\) with \(\|B_k\|\leq M\) and \(\mu_1\in(0, 1]\). Recall the definition of \(\hat{\eta}\), combining with~\eqref{eq:evare}, we have
\begin{equation}\label{eq:uppercg}
\|\cG(x_k)\| \leq (2 + M + c)\|\hat{x}_{k} - x_{k}\| + \|\ve_k\| \leq \hat{\eta}\|d_{k}\|. 
\end{equation}
This completes the proof of the lemma. \qed
\end{proof}

The following lemma shows that line search condition in Algorithm~\ref{alg:pnewton} is well-defined. 
\begin{lemma}
Let \(\alpha_k\) be chosen by the backtracking line search~\eqref{eq:ls} in Algorithm~\ref{alg:pnewton} at iteration \(k\). Then we have the step size estimate  
\begin{equation}\label{eq:alpk}
\alpha_k \geq \frac{\theta(c -\tau)}{L_H}
\end{equation}
with the cost function decrease satisfying
\begin{equation}\label{eq:dvarphi}
\varphi(x_{k+1}) - \varphi(x_k) \leq -\frac{\tau\theta(c - \tau)}{2L_H\hat{\eta}^2}\|\cG(x_k)\|^2.
\end{equation}
\end{lemma}
\begin{proof}
Denote \(l_k(x) := l_k(x; x_k)\). Recall the definition of \(l_k(x; x_k)\), we have \(l_k(x_k) = \varphi(x_k) = q_k(x_k)\) and  \(q_k(x) = l_k(x) + \frac{1}{2}(x - x_k)^\top H_k(x - x_k)\). By Lemma~\ref{lemma:dqkng} and the definition of \(H_k\), we have 
\begin{align}\label{eq:dlk}
0 \geq& q_k(\hat{x}_k) - \varphi(x_k) - \frac{\mu_2}{2}\min\{1, \|\cG(x_k)\|^{\delta}\}\|d_k\|^2 \nonumber \\
=& l_k(\hat{x}_k) - l_k(x_k) + \frac{1}{2}d_k^\top H_kd_k - \frac{\mu_2}{2}\min\{1, \|\cG(x_k)\|^{\delta}\}\|d_k\|^2\\
\geq & l_k(\hat{x}_k) - l_k(x_k) + \frac{c}{2}\|d_k\|^2, \nonumber 
\end{align}
which implies 
\begin{equation}\label{eq:dl}
l_k(x_k) - l_k(\hat{x}_k) \geq \frac{c}{2}\|d_k\|^2.
\end{equation} 
Notice that for any \(t\in[0, 1]\), 
\begin{align*}
\varphi(x_k) \!-\! \varphi(x_k + td_k) =& l_k(x_k) \!-\! f(x_k \!+\! td_k) \!-\! g(x_k \!+\! td_k)\\
=& l_k(x_k) \!-\! l_k(x_k \!+\! td_k) \!-\! (f(x_k \!+\! td_k) \!-\! f(x_k) \!-\! t\nabla f(x_k)^\top d_k)\\
\geq& l_k(x_k) \!-\! l_k(x_k \!+\! td_k)  \!-\! \frac{L_H}{2}t^2\|d_k\|^2,
\end{align*}
where the inequality holds since \(\nabla f\) is \(L_H\)-Lipschitz continuous. 

Hence, we have 
\begin{align*}
\varphi(x_k) - \varphi(x_k + td_k) - \frac{\tau }{2}t\|d_k\|^2\geq& l_k(x_k) - l_k(x_k + td_k)  - \frac{L_H}{2}t^2\|d_k\|^2  - \frac{\tau }{2}t\|d_k\|^2\\
\geq& t(l_k(x_k) - l_k(\hat{x}_k)) - \frac{L_H}{2}t^2\|d_k\|^2  - \frac{\tau }{2}t\|d_k\|^2\\
\geq& \frac{c}{2}t\|d_k\|^2 - \frac{L_H}{2}t^2\|d_k\|^2  - \frac{\tau }{2}t\|d_k\|^2\\
=&\frac{1}{2}((c - \tau) - L_Ht) t\|d_k\|^2,
\end{align*}
where the second inequality holds since \(l_k\) is convex and the last inequality holds up to~\eqref{eq:dl}. Therefore,~\eqref{eq:ls} holds for any \(t\) satisfies 
\[
0 < t\leq \frac{c - \tau}{L_H}. 
\] 
Combing with the backtracking technique used in Algorithm~\ref{alg:pnewton}, we have \(\alpha_k \geq \frac{\theta(c - \tau)}{L_H}\). Therefore,  
\[
\varphi(x_k) - \varphi(x_k + \alpha_k d_k) \geq \frac{\tau }{2}\alpha_k\|d_k\|^2 \geq \frac{\tau\theta(c - \tau)}{2L_H}\|d_k\|^2 \geq \frac{\tau\theta(c - \tau)}{2L_H\hat{\eta}^2}\|\cG(x_k)\|^2,
\]
where the last inequality follows from Lemma~\ref{lem:ngk}. 
This completes the proof of the lemma.  \qed
\end{proof}

Let \(\{x_k\}\) be the sequence generated by Algorithm~\ref{alg:pnewton}  from a starting point \(x_0\in{\rm dom}(g)\) with parameters that satisfy the requirements listed in the top line of Algorithm~\ref{alg:pnewton}. Denote \(\omega(x_0)\) as the cluster points set of \(\{x_k\}\). 
\begin{theorem}\label{th:limitsppn}
Suppose Assumption~\ref{assume:ncp} holds. Let \(\{x_k\}\) be the sequence generated by Algorithm~\ref{alg:pnewton}. 
\begin{enumerate}
\item[(a)] \(\omega(x_0)\subseteq \mathcal{S}^*\) is nonempty and compact. 
\item[(b)] Define \(\tilde{\eta} = \frac{2L_H\hat{\eta}^2}{\tau\theta(c - \tau)}\), we have 
\[
\min_{0\leq j\leq k}\|\cG(x_j)\| \leq \frac{\sqrt{\tilde{\eta}(\varphi(x_0) - \varphi_*)}}{\sqrt{k + 1}}. 
\] 
\end{enumerate}
\end{theorem}

\begin{proof}
(a) Under Assumption~\ref{assume:ncp},~\eqref{eq:dvarphi} implies that \(\{x_k\}\subseteq \mathcal{L}_{\varphi}(x_0)\), \(\{\varphi(x_k)\}\) converges and
\[
\lim_{k\to\infty}\|\mathcal{G}(x_k)\| = 0. 
\]
Hence, \(\omega(x_0) \neq \emptyset\). 
The continuity of \(\cG\) ensures the closedness of \(\omega(x_0)\) and \(\|\cG(\bar{x})\| = 0\) for any \(\bar{x}\in\omega(x_0)\).   

(b) Summing the inequality~\eqref{eq:dvarphi} over \(j = 0, 1, \cdots, k\), we obtain
\begin{align*}
\varphi(x_0) - \varphi_* \geq& \varphi(x_0) - \varphi(x_{k+1}) = \sum_{j=0}^k\varphi(x_j) - \varphi(x_{j+1}) \\
\geq& \frac{\tau\theta(c - \tau)}{2L_H\hat{\eta}^2}\sum_{j=0}^k\|\cG(x_j)\|^2 \geq  \frac{k+1}{\tilde{\eta}}\min_{0\leq j\leq k}\|\cG(x_j)\|^2. 
\end{align*}
So the statement (b) holds. \qed
\end{proof}

\begin{remark}
Here are some remarks for Algorithm~\ref{alg:pnewton}. 
\begin{itemize}
\item[(i)] For any \(\epsilon > 0\), denote \(\overline{K} := \lceil \tilde{\eta}(\varphi(x_0) - \varphi_*)\epsilon^{-2}\rceil\). Some iterate \(x_k\), \(k = 0, \cdots, \overline{K} + 1\) generated by Algorithm~\ref{alg:pnewton} will satisfy \(\|\mathcal{G}(x_k)\|\leq \epsilon\). Suppose that \(\|\cG(x_{k+1})\| > \epsilon\) for all \(k = 0, 1, \cdots, \overline{K}\). From~\eqref{eq:dvarphi}, we have 
\begin{align*}
\varphi(x_0) - \varphi(x_{\overline{K}+1}) =& \sum_{l=0}^{\overline{K}}\varphi(x_l) - \varphi(x_{l+1}) \geq  \frac{\tau\theta(c - \tau)}{2L_H\hat{\eta}^2}\sum_{l=0}^{\overline{K}}\|\cG(x_l)\|^2 \\
 \geq& \frac{(\overline{K}+1)\epsilon^2}{\tilde{\eta}}> \varphi(x_0) - \varphi_*,
\end{align*}
where the last inequality follows from the definition of \(\overline{K}\). The above inequality contradicts the definition of {\(\varphi_*\)}, so the statement holds. 

\item[(ii)] If \(g(x) \equiv 0\) and \(H_k = \nabla^2f(x_k)  + \varsigma_kI\), where \(\varsigma_k = \max\{0, -\lambda_{\min}(\nabla^2f(x_k))\} + c + \mu_1\min\{1, \|\nabla f(x_k)\|^{\delta}\}\), then Algorithm~\ref{alg:pnewton} reduces to the regularized Newton method for nonconvex problem \(\min_x f(x)\) with the update formula \(x_{k+1} = x_k + \alpha_k d_k\), where \(d_k\) satisfies 
\[
\|(\nabla^2f(x_k) + \varsigma_kI)d_k + \nabla f(x_k)\| \leq \frac{\mu_2}{2}\min\{1, \|\nabla f(x_k)\|^{\delta}\}\|d_k\|
\]
is an inexact solution of the regularized Newton equation \((\nabla^2f(x_k) + \varsigma_kI)d = -\nabla f(x_k)\) and \(\alpha_k\) satisfies line search condition 
\[
f(x_k + \alpha_kd_k) \leq f(x_k) - \frac{\tau}{2}\alpha_k\|d_k\|^2. 
\]
According to the statement (i) of this remark, Algorithm~\ref{alg:pnewton} returns a point satisfies \(\|\nabla f(x_k)\| \leq \epsilon\) within at most \(\mathcal{O}(\epsilon^{-2})\) iterations. The global complexity bound \(\mathcal{O}(\epsilon^{-2})\) is consistent with the regularized Newton method proposed by~Ueda and Yamashita~\cite{UY10}, which requires \(d_k\) to be an exact solution of the regularized Newton equation, \(\delta \in [0, \frac{1}{2}]\), and \(\alpha_k\) satisfies \(f(x_k + \alpha_kd_k) \leq f(x_k) - \frac{\tau}{2}\alpha_k\nabla f(x_k)^\top d_k\). 
\end{itemize}
\end{remark}

\subsection{Local convergence under metric subregularity property}
The metric subregularity property of subgradient mappings, also known as the second-order growth condition, has been used in the literature~\cite{AG08,AG15,DMN14} to analyze the local convergence rate for nonsmooth optimization. Let \({\rm gph}F:=\{(x, y)\in X\times Y\vert y\in F(x)\}\) be the graph of the set-valued mapping \(F\) and \(F^{-1}(y) := \{x\in X\vert y \in F(x)\}\) be the inverse of \(F\). Denote \(\mathbb{B}(x, r)\) as the open Euclidean norm ball centered at \(x\) with radius \(r > 0\). The metric \(q\)-subregularity is defined as follows. 

\begin{definition}\cite[Def. 3.1]{MO15}
Let \(F: X \rightrightarrows Y\) with \((\bar{x}, \bar{y})\in {\rm gph} F\), and let \(q > 0\). We say that \(F\) is metrically \(q\)-subregular at \((\bar{x}, \bar{y})\) if there are constants \(\kappa\), \(\epsilon > 0\) such that 
\[
{\rm dist}(x; F^{-1}(\bar{y})) \leq \kappa {\rm dist}^q(\bar{y}; F(x)), \quad {\rm for~all}~x\in \mathbb{B}(\bar{x}, \epsilon). 
\]
\end{definition}

The higher-order metric subregularity of subgradient mappings with \(q>1\) was first studied in~\cite{MO15}, and was  applied therein to give the local convergence analysis of the Newton-type methods~\cite{MO15,MYZZ22,LPWY23}. Examples with the  higher-order metric subregularity property, see~\cite{MO15}. 
The equivalence of the \(q\)-subregularity of the subgradient mapping and the residual mapping was given in~\cite{LPWY23}. 

Inspired by the work of~\cite{LPWY23}, we give the local superlinear convergence rate of the sequence generated by Algorithm~\ref{alg:pnewton} with 
\[
H_k = \nabla^2f(x_k) + ([-\lambda_{\min}(\nabla^2f(x_k))]_+ + c + \mu_1\min\{1, \|\mathcal{G}(x_k)\|^{\delta}\})I,
\] 
where \([-\lambda_{\min}(\nabla^2f(x_k))]_+ = \max\{0, -\lambda_{\min}(\nabla^2f(x_k))\}\). 
In this section, we assume that the residual mapping \(\mathcal{G}\) satisfies the metric \(q\)-subregularity property. 

 \begin{assumption}\label{assume:errorbound}
For any \(\bar{x}\in\omega(x_0)\), the metric \(q\)-subregularity at \(\bar{x}\) with \(q > 1\) on \(\mathcal{S}^*\) holds, that is, there exist \(\epsilon > 0\) and \(\kappa > 0\) such that 
\[
{\rm dist}(x, \mathcal{S}^*) \leq \kappa \|\mathcal{G}(x)\|^q, \quad \forall x\in\mathbb{B}(\bar{x}, \epsilon). 
\]
\end{assumption}

We also assume that \(f\) and \(g\) satisfy the following assumption.
\begin{assumption}\label{assume:g}
\begin{itemize}
\item[(i)] \(f: \mathbb{R}^n\to(-\infty, +\infty]\) is twice continuously differentiable on an open set \(\Omega_2\) containing the effective domain \({\rm dom}(g)\) of \(g\), \(\nabla f\) is \(L_H\)-Lipschitz continuous over \(\Omega_2\); \(\nabla^2 f\) is \(L_C\)-Lipschitz continuous over an open neighborhood of \(\omega(x_0)\) with radius \(\epsilon_0\) for some \(\epsilon_0 > 0\). 
\item[(ii)] \(g: \mathbb{R}^n\to(-\infty, +\infty]\) is proper closed convex, nonsmooth and continuous.
\item[(iii)] For any \(x_0 \in {\rm dom}(g)\), the level set \(\mathcal{L}_{\varphi}(x_0) = \{x\vert \varphi(x) \leq \varphi(x_0)\}\) is bounded. 
\end{itemize}
\end{assumption}
Under Assumption~\ref{assume:g} (i), we have \([-\lambda_{\min}(\nabla^2f(x_k))]_+ \leq L_H\).  Under Assumption~\ref{assume:g} (i) and (ii), \(\nabla f(\cdot) + \partial g(\cdot)\) is outer semicontinuous over \({\rm dom}(g)\)~\cite{RW04}. Hence, the stationary set \(\mathcal{S}^*\) is closed. 
The following result holds from the continuity of \(\varphi\). 
\begin{lemma}
\(\varphi \equiv \varphi_*:=\lim_{k\to\infty}\varphi(x_k)\) on \(\omega(x_0)\). 
\end{lemma}
Denote \(\bar{x}_k = \argmin_x\{q(x; x_k, H_k)\}\). The following lemma estimates the error bound between \(\bar{x}_k\) and \(\hat{x}_k\). 
\begin{lemma}\label{lem:42}
For each \(k\in\mathbb{N}\), it holds that 
\begin{equation}\label{eq:errhbxk}
\|\hat{x}_k - \bar{x}_k\| \leq \frac{\mu_2}{2c}(1 + \|H_k\|)\|d_k\|.
\end{equation}
\end{lemma}
\begin{proof}
By the definition of \(\bar{x}_k\) and using the first-order optimality condition, we have 
\begin{equation}\label{eq:optbarx}
-\nabla f(x_k) - H_k(\bar{x}_k - x_k) \in\partial g(\bar{x}_k). 
\end{equation}
Combining with~\eqref{eq:ek2}, using the monotonicity of \(\partial g\), we have 
\begin{align*}
0 \leq& \langle \hat{x}_k - \bar{x}_k - r_k(\hat{x}_k), r_k(\hat{x}_k) - H_k(\hat{x}_k - \bar{x}_k)\rangle\\
\leq & \langle (\hat{x}_k - \bar{x}_k) + H_k(\hat{x}_k - \bar{x}_k), r_k(\hat{x}_k)\rangle -  \langle \hat{x}_k - \bar{x}_k, H_k( \hat{x}_k - \bar{x}_k)\rangle. 
\end{align*} 
Notice that from the definition of \(H_k\), we have \(H_k \succeq c I\), that is, \(H_k - cI\) is positive semidefinite. Therefore, the above inequality implies that
\[
c\|\hat{x}_k - \bar{x}_k\| \leq (1 + \|H_k\|)\|r_k(\hat{x}_k)\| \leq \frac{\mu_2}{2}(1 + \|H_k\|)\|d_k\|,
\]
where the last inequality follows from~\eqref{eq:evare} and~\eqref{eq:errvek}. 
The statement holds. \qed
\end{proof}
Next, we estimate the error bound between \(x_k\) and \(\bar{x}_k\) in terms of \({\rm dist}(x_k, \mathcal{S}^*)\). 

\begin{lemma}\label{lemma:43}
Consider any \(\bar{x}\in\omega(x_0)\). Suppose that Assumption~\ref{assume:g} holds. Then, for all \(x_k \in \mathbb{B}(\bar{x}, \epsilon_0/2)\) with \(\epsilon_0\) defined in Assumption~\ref{assume:g} (i),  
\[
\|x_k - \bar{x}_k\| \leq \frac{L_C}{2c}{\rm dist}^2(x_k, \mathcal{S}^*) + \frac{1}{c}([-\lambda_{\min}(\nabla^2f(x_k))]_+ + 2c + \mu_1){\rm dist}(x_k, \mathcal{S}^*).
\]
\end{lemma}
\begin{proof}
For any \(x_k\in\mathbb{B}(\bar{x}, \epsilon_0/2)\), let \(\Pi_{\mathcal{S}^*}(x_k) \triangleq \{x\in\mathcal{S}^*\vert {\rm dist}(x, x_k) \leq {\rm dist}(z, x_k),~\forall z\in\mathcal{S}^*\}\) be the projection set of \(x_k\) onto \(\mathcal{S^*}\). Then \(\Pi_{\mathcal{S}^*}(x_k) \neq \emptyset\) since \(\mathcal{S}^*\) is closed. Pick \(x_{k, *}\in \Pi_{\mathcal{S}^*}(x_k)\). Notice that \(\bar{x} \in\omega(x_0)\subseteq \mathcal{S}^*\), we have 
\[
\|x_{k, *} -\bar{x}\| \leq \|x_{k, *} - x_k\| + \|x_k - \bar{x}\| \leq 2\|x_k - \bar{x}\| \leq \epsilon_0,
\]
which implies that \(x_{k, *}\in\mathbb{B}(\bar{x}, \epsilon_0)\). Hence, \((1- t)x_k + tx_{k, *}\in\mathbb{B}(\bar{x}, \epsilon_0)\cap {\rm dom}(g)\) for all \(t\in[0,1]\). Notice that \(x_{k, *} \in\mathcal{S}^*\), we have \(-\nabla f(x_{k, *}) \in \partial g(x_{k, *})\). Combine with~\eqref{eq:optbarx}, using the monotonicity of \(\partial g\), we have 
\[
0 \leq \langle x_{k, *} - \bar{x}_k, -\nabla f(x_{k, *}) + \nabla f(x_k) + H_k(x_{k, *} - x_k)\rangle + \langle x_{k, *} - \bar{x}_k, H_k(\bar{x}_k - x_{k, *})\rangle.
\]
Using \(H_k \succeq cI\) again, we have
\begin{align*}
\|\bar{x}_k - x_{k, *}\| \leq& \frac{1}{c}\| \nabla f(x_k) -\nabla f(x_{k, *})  + H_k(x_{k, *} - x_k)\|\\
=& \frac{1}{c}\|\int_0^1[H_k - \nabla^2f(x_k + t(x_{k,*} - x_k))](x_{k, *} - x_k)dt\|\\
\leq & \frac{1}{c}\|\int_0^1[\nabla^2f(x_k) - \nabla^2f(x_k + t(x_{k,*} - x_k))](x_{k, *} - x_k)dt\| \\
&+\frac{1}{c} ([-\lambda_{\min}(\nabla^2f(x_k))]_+ + c + \mu_1\min\{1, \|\mathcal{G}(x_k)\|^{\delta}\})\|x_{k, *} - x_k\|\\
\leq & \frac{L_C}{2c}\|x_{k, *} - x_k\|^2 + \frac{1}{c}([-\lambda_{\min}(\nabla^2f(x_k))]_+ + c + \mu_1)\|x_{k, *} - x_k\|, 
\end{align*}
where the second inequality follows from the definition of \(H_k\) and the last inequality follows from Assumption~\ref{assume:g} (i). 
Therefore, 
\begin{align*}
\|x_k - \bar{x}_k\| \leq& \|x_k - x_{k, *}\| + \|x_{k, *} - \bar{x}_k\|\\
\leq& \frac{L_C}{2c}\|x_{k, *} - x_k\|^2 + \frac{1}{c}([-\lambda_{\min}(\nabla^2f(x_k))]_+ + 2c + \mu_1)\|x_{k, *} - x_k\|.
\end{align*}
The statement holds. \qed
\end{proof}

The following lemma shows that the unit step-size will occur when the iterates are close enough to a cluster point. 

\begin{lemma}\label{lemma:45}
Suppose that Assumption~\ref{assume:g} holds. Let \(\{x_k\}\) be the sequence generated by Algorithm~\ref{alg:pnewton} with \(c > \frac{\mu_2(1 + 2L_H + \mu_1)}{2 - \mu_2}\). 
For any \(\bar{x}\in\omega(x_0)\), there exists \(\bar{k}\in\mathbb{N}\) such that 
\[
x_{k+1} = \hat{x}_k, \quad {\rm for~all}~x_k\in\mathbb{B}(\bar{x}, \epsilon_0/2)~{\rm with}~k \geq \bar{k}, 
\]
where \(\epsilon_0\) is given in Assumption~\ref{assume:g}. 
\end{lemma}
\begin{proof}
By invoking Lemmas~\ref{lem:42} and \ref{lemma:43}, \([-\lambda_{\min}(\nabla^2f(x_k))]_+ \leq L_H\), and \(\|H_k\| \leq 2L_H + \mu_1 + c\), for all \(x_k\in\mathbb{B}(\bar{x}, \epsilon_0/2)\), we have
\begin{align*}
&\|d_k\| = \|\hat{x}_k - x_k\| \leq \|\hat{x}_k - \bar{x}_k\| + \|\bar{x}_k - x_k\|\\
\leq& \frac{\mu_2}{2c}(1 \!+\! \|H_k\|)\|d_k\| \!+\!\! \frac{L_C}{2c}\!{\rm dist}^2(x_k, \mathcal{S}^*\!) \!+\! \frac{1}{c}([-\!\lambda_{\min}(\nabla^2f(x_k))]_+ \!+\! 2c \!+\! \mu_1){\rm dist}(x_k, \mathcal{S}^*\!)\\
\leq & \frac{\mu_2}{2c}(1 + 2L_H + \mu_1 + c)\|d_k\| + \frac{L_C}{2c}{\rm dist}^2(x_k, \mathcal{S}^*) + \frac{1}{c}(L_H + 2c + \mu_1){\rm dist}(x_k, \mathcal{S}^*).
\end{align*}
Notice that  the above inequality implies that
\begin{equation}\label{eq:ndk}
\|d_k\| \leq \frac{L_C}{\tilde{\tilde{\eta}}}{\rm dist}^2(x_k, \mathcal{S}^*) + \frac{2(L_H + 2c + \mu_1)}{\tilde{\tilde{\eta}}}{\rm dist}(x_k, \mathcal{S}^*),
\end{equation}
where \(\tilde{\tilde{\eta}} = 2c - \mu_2(c + 1 + \mu_1 + 2L_H) > 0\). Therefore, \(\|d_k\| = \mathcal{O}({\rm dist}(x_k, \mathcal{S}^*))\), which implies that there exists \(\bar{k} \in \mathbb{N}\), such that \(\|d_k\| \leq {\min\{\frac{3(c - \tau)}{L_C}, \frac{\epsilon_0}{2}\}}\) for all \(k \geq \bar{k}\). Notice that \(\|x_k + d_k - \bar{x}\| \leq \|x_k - \bar{x}\| + \|d_k\| < \frac{\epsilon_0}{2} + \frac{\epsilon_0}{2}  = \epsilon_0\), combing with \(\bar{x}\in\omega(x_0)\), we know that \(x_k + d_k\) belongs to the open neighborhood of \(\omega(x_0)\) with radius \(\epsilon_0\).

Notice that 
\begin{align*}
&\varphi(x_k + \theta^jd_k) - \varphi(x_k)\\
=& f(x_k + \theta^j d_k) - f(x_k) + g(x_k + \theta^jd_k) - g(x_k) \\
\leq & \int_0^1\theta^j\langle \nabla f(x_k + t\theta^jd_k) - \nabla f(x_k), d_k\rangle dt + \theta^j\langle\nabla f(x_k), d_k\rangle + \theta^j[g(\hat{x}_k) - g(x_k)]\\
=&\int_0^1\int_0^1\!\!t\theta^{2j}\langle d_k, [\nabla^2f(x_k + st\theta^jd_k) - \nabla^2f(x_k)]d_k\rangle dsdt \\
&+\!\!\int_0^1\! \!t\theta^{2j}\langle d_k, \nabla^2f(x_k)d_k\rangle dt \!+\! \theta^j\langle\nabla f(x_k), d_k\rangle \!+\!\theta^j[l_k(\hat{x}_k)\!-\! l_k(x_k) \!-\! \langle \nabla f(x_k), d_k\rangle]\\
\leq& \frac{L_C}{6}\theta^{3j}\|d_k\|^3 + \frac{1}{2}\theta^{2j}\langle d_k, \nabla^2f(x_k)d_k\rangle +\theta^j[l_k(\hat{x}_k) - l_k(x_k) ]\\
\leq& \frac{L_C}{6}\theta^{j}\|d_k\|^3  + \frac{1}{2}\theta^{j}\langle d_k, (\nabla^2f(x_k) + [-\lambda_{\min}(\nabla^2f(x_k))]_+I)d_k\rangle\\
& + \frac{\theta^j}{2}(\mu_2\min\{1, \|\mathcal{G}(x_k)\|^\delta\}\|d_k\|^2 - d_k^\top H_kd_k)\\
=& \frac{L_C}{6}\theta^{j}\|d_k\|^3 - \frac{1}{2}\theta^j(c + (\mu_1 - \mu_2)\min\{1, \|\mathcal{G}(x_k)\|^{\delta}\})\|d_k\|^2, 
\end{align*}
where the first inequality follows from the convexity of \(g\), the second inequality follows from Assumption~\ref{assume:g} (i), the last inequality follows from~\eqref{eq:dlk}, and \(\theta\in(0,1)\), and the last equality follows from the definition of \(H_k\). Hence, 
\begin{align*}
&\varphi(x_k + \theta^jd_k) - \varphi(x_k) + \frac{\tau\theta^j}{2}\|d_k\|^2 \\
\leq& \frac{L_C}{6}\theta^{j}\|d_k\|^3  + \frac{\theta^{j}}{2}((\mu_2-\mu_1)\min\{1, \|\mathcal{G}(x_k)\|^{\delta}\} - c + \tau)\|d_k\|^2  \\
\leq& -\frac{\theta^j}{2}\|d_k\|^3(\frac{c - \tau}{\|d_k\|} - \frac{L_C}{3}).
\end{align*}
Recall that \(\frac{c - \tau}{\|d_k\|} - \frac{L_C}{3} \geq 0\) for all \(k \geq \bar{k}\) and \(x_k\in\mathbb{B}(\bar{x}, \epsilon_0/2)\). Therefore, for any \(k \geq \bar{k}\) with \(x_k\in\mathbb{B}(\bar{x}, \epsilon_0/2)\), line search condition~\eqref{eq:ls} holds with \(j = 0\). The statement holds.  \qed
\end{proof}

\begin{theorem}\label{th:lcrxk}
Suppose that Assumptions~\ref{assume:errorbound} and~\ref{assume:g} hold. Let \(\{x_k\}\) be the sequence generated by Algorithm~\ref{alg:pnewton} with \(c > \frac{\mu_2(1 + 2L_H + \mu_1)}{2 - \mu_2}\). 
Then for any \(\bar{x}\in\omega(x_0)\), \(\{x_k\}\) converges to \(\bar{x}\) with the Q-superlinear convergence rate of order \(q\). 
\end{theorem}
\begin{proof}
Notice that under Assumption~\ref{assume:errorbound} and the fact \(\lim_{k\to\infty}\|\mathcal{G}(x_k)\| = 0\), the statement of Lemma~\ref{lemma:45}, and \eqref{eq:ndk}, there exists \(\hat{k}\in \mathbb{N}\), such that for all \(k \geq \hat{k}\), \(\|\mathcal{G}(x_k)\| \leq 1\), \(\|d_k\| \leq c_1{\rm dist}(x_k, \mathcal{S}^*)\) for some \(c_1 > 0\), and \(x_{k+1} = \hat{x}_k\) if \(x_k \in \mathbb{B}(\bar{x}, \epsilon_1)\) with \(\epsilon_1 = \min\{\epsilon, \epsilon_0/2\}\). 

We first show that for all \(k \geq \hat{k}\), if \(x_k \in \mathbb{B}(\bar{x}, \epsilon_1)\) and \(x_{k+1} = \hat{x}_k \in \mathbb{B}(\bar{x}, \epsilon_1)\), then 
\begin{equation}\label{eq:distx}
{\rm dist}(x_{k+1}, \mathcal{S}^*) = o({\rm dist}(x_k, \mathcal{S}^*)). 
\end{equation}
From Assumption~\ref{assume:errorbound}, for \(x_{k+1} = \hat{x}_k\), we have
\begin{align*}
{\rm dist}(x_{k+1}, \mathcal{S}^*) \leq& \kappa [\|\mathcal{G}(x_{k+1})\| - \|r_k(x_{k+1})\| + \|r_k(\hat{x}_k)\|]^q\\
\leq& \kappa [\|\mathcal{G}(x_{k+1})\| - \|r_k(x_{k+1})\| + \mu_2\|d_k\|]^q\\
\leq& \kappa [\|\mathcal{G}(x_{k+1})\| - \|r_k(x_{k+1})\| + \mu_2c_1{\rm dist}(x_k, \mathcal{S}^*)]^q,
\end{align*}
where the second inequality follows from~\eqref{eq:evare} and~\eqref{eq:errvek}. 
Notice that 
\begin{align*}
&\|\mathcal{G}(x_{k+1})\| - \|r_k(x_{k+1})\| \leq \|\mathcal{G}(x_{k+1}) - r_k(x_{k+1})\|\\
=& \|{\rm prox}_g(x_{k+1} - (\nabla f(x_k) + H_k(x_{k+1} - x_k))) - {\rm prox}_g(x_{k+1} - \nabla f(x_{k+1}))\|\\
\leq& \|\nabla f(x_{k+1}) - \nabla f(x_k) - H_k(x_{k+1} - x_k)\|\\
\leq& \|\int_0^1[\nabla^2f(x_k + t(x_{k+1} - x_k)) - \nabla^2f(x_k)](x_{k+1} - x_k)dt\| \\
&+ ([-\lambda_{\min}(\nabla^2f(x_k))]_+ + c + \mu_1)\|x_{k+1} - x_k\|\\
\leq& \frac{L_C}{2}\|d_k\|^2 +  (L_H + c + \mu_1)\|d_k\|\\
\leq& [\frac{c_1^2L_C}{2}{\rm dist}(x_k, \mathcal{S}^*) + c_1(L_H + c + \mu_1)]{\rm dist}(x_k, \mathcal{S}^*),
\end{align*}
where first inequality follows from the nonexpansivity of \({\rm Prox}_g\) and the third inequality follows from Assumption~\ref{assume:g} (i). 
Therefore, denote \(\tilde{c}_1 = L_H + c + \mu_1 + \mu_2\), we have 
\begin{align}\label{eq:xks}
{\rm dist}(x_{k+1}, \mathcal{S}^*) \leq&\kappa \left[\frac{c_1^2L_C}{2}{\rm dist}^2(x_k, \mathcal{S}^*) + c_1\tilde{c}_1{\rm dist}(x_k, \mathcal{S}^*)\right]^q,
\end{align}
which implies~\eqref{eq:distx} holds since \(q > 1\) and \(\lim_{k\to\infty}{\rm dist}(x_k, \mathcal{S}^*) = 0\). Therefore, for any \(c_2\in(0, 1)\), there exist \(0 < \epsilon_2 < \epsilon_1\) and \(\tilde{k} \geq \hat{k}\), such that for all \(k \geq \tilde{k}\), if \(x_k\in\mathbb{B}(\bar{x}, \epsilon_2)\) and \(x_{k+1} = \hat{x}_k\in\mathbb{B}(\bar{x}, \epsilon_2)\), then we have
\begin{equation}\label{eq:distkk}
{\rm dist}(x_{k+1}, \mathcal{S}^*) \leq c_2{\rm dist}(x_k, \mathcal{S}^*).
\end{equation}
Let \(\bar{\epsilon} := \min\{\frac{\epsilon_2}{2}, \frac{(1 - c_2)\epsilon_2}{2c_1}\}\). Next, we show that if \(x_{k_0}\in\mathbb{B}(\bar{x}, \bar{\epsilon})\) for some \(k_0 \geq \tilde{k}\), then \(x_{k_0+1} = \hat{x}_{k_0}\in \mathbb{B}(\bar{x}, \epsilon_2)\) for all \(k \geq k_0\) by induction. 

Notice that \(\bar{x} \in\omega(x_0)\), there exists \(k_0 \geq \tilde{k}\), such that \(x_{k_0}\in\mathbb{B}(\bar{x}, \bar{\epsilon})\). Therefore, 
\begin{align*}
\|\hat{x}_{k_0} - \bar{x}\| \leq& \|x_{k_0} - \bar{x}\| + \|x_{k_0} - \hat{x}_{k_0}\| = \|x_{k_0} - \bar{x}\| + \|d_{k_0}\| \\
\leq& \|x_{k_0} - \bar{x}\| + c_1{\rm dist}(x_{k_0}, \mathcal{S}^*) \leq \|x_{k_0} - \bar{x}\| + c_1\|x_{k_0} - \bar{x}\| \\
\leq& (1 + c_1)\bar{\epsilon}\leq \epsilon_2. 
\end{align*}
By Lemma~\ref{lemma:45}, we have \(x_{k_0 + 1} = \hat{x}_{k_0}\) and \(x_{k_0 + 1} \in\mathbb{B}(\bar{x}, \epsilon_2)\). 

For any \(k > k_0\), suppose that for all \(k_0\leq l\leq k-1\), we have \(x_{l+1} = \hat{x}_l\in\mathbb{B}(\bar{x}, \epsilon_2)\). Then, 
\begin{align*}
\|x_{k+1} - x_{k_0}\| \leq& \sum_{l=k_0}^k\|d_l\| \leq c_1\sum_{l=k_0}^k{\rm dist}(x_l, \mathcal{S}^*) \overset{\eqref{eq:distkk}}{\leq} c_1\sum_{l=k_0}^kc_2^{l-k_0}{\rm dist}(x_{k_0}, \mathcal{S}^*)\\
\leq& \frac{c_1}{1 - c_2}\|x_{k_0} - \bar{x}\|. 
\end{align*}
Therefore, \(\|x_{k+1} - \bar{x}\| \leq \|x_{k+1} - x_{k_0}\| + \|x_{k_0} - \bar{x}\| \leq (1 + \frac{c_1}{1 - c_2})\|x_{k_0} - \bar{x}\| \leq (1 + \frac{c_1}{1 - c_2})\bar{\epsilon} \leq \epsilon_2\). Hence, \(x_{k+1} \in\mathbb{B}(\bar{x}, \epsilon_2)\). 

Next, we show that \(\{x_k\}\) is a Cauchy sequence. For any \(\varepsilon > 0\), there exists \(\bar{\bar{k}} \geq k_0\), such that 
\[
{\rm dist}(x_k, \mathcal{S}^*) < \tilde{\varepsilon}, \quad \forall k > \bar{\bar{k}}
\]
since \(\lim_{k\to\infty}{\rm dist}(x_k, \mathcal{S}^*) = 0\), where \(\tilde{\varepsilon} = \frac{1 - c_2}{c_1}\varepsilon\). For any \(k_1, k_2 >  \bar{\bar{k}}\), where without loss of generality we assume \(k_1 > k_2\), we have 
\begin{align*}
\|x_{k_1} - x_{k_2}\| \leq& \sum_{j = k_2}^{k_1 -1}\|x_{j+1} - x_j\| = \sum_{j = k_2}^{k_1 -1}\|d_j\| \leq c_1\sum_{j=k_2}^{k_1 - 1}{\rm dist}(x_j, \mathcal{S}^*)\\
\overset{\eqref{eq:distkk}}{\leq}& c_1\sum_{j=k_2}^{k_1 - 1}c_2^{j - k_2}{\rm dist}(x_{k_2}, \mathcal{S}^*) \leq \frac{c_1}{1 - c_2}{\rm dist}(x_{k_2}, \mathcal{S}^*)< \frac{c_1}{1 - c_2}\tilde{\varepsilon} =\varepsilon.
\end{align*} 
Recall Theorem~\ref{th:limitsppn} (a), the cluster point set \(\omega(x_0)\subseteq \mathbb{R}^n\) of \(\{x_k\}\) is closed. Hence, the Cauchy sequence \(\{x_k\}\) converges to some \(\bar{x}\in\omega(x_0)\). By passing the limit \(k_1\to\infty\) to the above formula and using~\eqref{eq:xks}, we have for any \(k > \bar{\bar{k}}\), 
\[
\|x_{k+1} - \bar{x}\|\leq \frac{c_1}{1 - c_2}{\rm dist}(x_{k+1}, \mathcal{S}^*) \leq \mathcal{O}({\rm dist}^q(x_k, \mathcal{S}^*)) \leq \mathcal{O}(\|x_k - \bar{x}\|^q).
\]
Therefore, \(\{x_k\}\) converges to \(\bar{x}\) with the Q-superlinear rate of order \(q\).  \qed
\end{proof}

\section{The inexact proximal Newton method when \(L_H\) is known}\label{appendix: lhisgiven}

In this section, we show that when the Lipschitz constant \(L_H\) is used to define \(H_k\), Algorithm~\ref{alg:pnewton} is well-defined without performing line search. 

Denote \(\widehat{H}_k = B_k + (L_H + c + \mu_1\min\{1, \|\mathcal{G}(x_k)\|^{\delta}\})I\) for some \(c > 0\) and \(\mu_1 > 0\). We present the inexact proximal Newton method without line search in Algorithm~\ref{alg:pmm}. 

\begin{algorithm*}[!h]
\caption{The inexact proximal Newton method without line search.}\label{alg:pmm}
\begin{algorithmic}[1]
\REQUIRE{\(x_0\), \(c > 0\), \(\mu_1\in(0, 1]\), \(\mu_2 \in (0, \mu_1]\) and \(\delta \in [0, 1]\)}
\FOR{\(k = 0, 1, \cdots, \)}
\STATE{compute}
\begin{equation*}
x_{k+1} = \argmin_{x}\{q(x; x_k, \widehat{H}_k) + \ve_k^\top(x - x_k)\},
\end{equation*}
where the error vector \(\ve_k\) satisfies \(\|\ve_k\| \leq \frac{\mu_2}{2}\min\{1, \|\cG(x_k)\|^{\delta}\}\|x_{k+1} - x_k\|\). 
\ENDFOR
\RETURN{\(\{x_k\}\)}
\end{algorithmic}
\end{algorithm*}

Similar results to Lemma~\ref{lem:ngk} and Theorem~\ref{th:limitsppn} hold for Algorithm~\ref{alg:pmm}.
\begin{theorem}\label{th:limitspwithoutls}
Suppose Assumption~\ref{assume:ncp} holds. Let \(\{x_k\}\) be the sequence generated by Algorithm~\ref{alg:pmm}. 
Denote \(\bar{\eta} = M + L_H + 3\). 
The following statements hold. 
\begin{itemize}
\item[(a)] \(\|\cG(x_k)\| \leq \bar{\eta}\|x_{k+1} - x_k\|\). 
\item[(b)] \(\varphi(x_k) - \varphi(x_{k+1}) \geq  \frac{c}{2}\|x_{k+1} - x_k\|^2\). 
\item[(c)] \(\omega(x_0)\subseteq \mathcal{S}^*\) is nonempty, closed, and connected.  
\item[(d)] Define \(\zeta = \frac{2\bar{\eta}^2}{c}\), we have 
\[
\min_{0\leq j\leq k}\|\cG(x_j)\| \leq \frac{\sqrt{\zeta(\varphi(x_0) - \varphi_*)}}{\sqrt{k + 1}}. 
\] 
\end{itemize}
\end{theorem}
\begin{proof}
(a) The proof of statement (a) is similar to the proof of Lemma~\ref{lem:ngk}.

(b) Notice that Lemma~\ref{lemma:dqkng} still holds and we have
\begin{align}\label{eq:dphi}
\varphi(x_k) =& q_k(x_k) \geq q_k(x_{k+1}) \!-\! \frac{\mu_2}{2}\min\{1, \|\cG(x_k)\|^{\delta}\}\|x_{k+1} \!-\! x_k\|^2 \nonumber\\
=& \varphi(x_{k+1}) \!-\! (f(x_{k+1}) \!-\! f(x_k) \!-\! \nabla f(x_k)^\top(x_{k+1}  \!-\! x_k)) \nonumber \\
&+\!  \frac{1}{2}(x_{k+1}  \!-\! x_k)^\top \widehat{H}_k(x_{k+1} \!-\! x_k) \!-\! \frac{\mu_2}{2}\min\{1, \|\cG(x_k)\|^{\delta}\}\|x_{k+1}  \!-\! x_k\|^2 \nonumber \\
\geq& \varphi(x_{k+1}) \!-\! \frac{L_H}{2}\|x_{k+1}  \!-\! x_k\|^2 \!+\! \frac{1}{2}(x_{k+1}  \!- \!x_k)^\top \widehat{H}_k(x_{k+1}  \!-\! x_k) \nonumber\\
&-\! \frac{\mu_2}{2}\min\{1, \|\cG(x_k)\|^{\delta}\}\|x_{k+1}  \!-\! x_k\|^2 \nonumber\\
\geq & \varphi(x_{k+1}) \!+\! \frac{1}{2}(c \!+\! (\mu_1 \!-\!\mu_2)\min\{1, \|\mathcal{G}(x_k)\|^{\delta}\})\|x_{k+1}  \!-\! x_k\|^2.
\end{align}

(c) Under Assumption~\ref{assume:ncp}, statement (b) implies that \(\{\varphi(x_k)\}\) converges and
\[
\lim_{k\to\infty}\|x_{k+1} - x_k\|^2 = 0. 
\]
Hence, \(\omega(x_0)\) is nonempty, closed, and connected by the Ostrowski condition~\cite[Th.~8.3.9]{FP03}. 

(d) Summing the inequality in statement (b) over \(j = 0, 1, \cdots, k\), we obtain
\begin{align*}
\varphi(x_0) - \varphi_* \geq& \varphi(x_0) - \varphi(x_{k+1}) = \sum_{j=0}^k\varphi(x_j) - \varphi(x_{j+1}) \\
\geq& \frac{c}{2}\sum_{j=0}^k\|x_{j+1} - x_j\|^2 \geq \frac{c}{2\bar{\eta}^2}\sum_{j=0}^k\|\cG(x_j)\|^2 \geq  \frac{k+1}{\zeta}\min_{0\leq j\leq k}\|\cG(x_j)\|^2,
\end{align*}
where the third inequality follows from statement (a). 
So the statement holds. \qed
\end{proof}

Similar to Theorem~\ref{th:lcrxk}, the local convergence rate for the sequence generated by Algorithm~\ref{alg:pmm} can be established under the metric \(q\)-subregularity property. 
\begin{theorem}
Suppose that Assumptions~\ref{assume:errorbound} and~\ref{assume:g} hold. Let \(\{x_k\}\) be the sequence generated by Algorithm~\ref{alg:pmm} with \(c > \frac{(3\mu_2 - 2)L_H + \mu_2}{2 - \mu_2}\). 
Then for any \(\bar{x}\in\omega(x_0)\), \(\{x_k\}\) converges to \(\bar{x}\) with the Q-superlinear convergence rate of order \(q\). 
\end{theorem}
\begin{remark}
Here are some remarks for Algorithm~\ref{alg:pmm}. 
\begin{itemize}
\item[(i)] The Kurdyka-\L ojasiewicz (KL) property~\cite{L93,K98} has been widely used to establish the convergence result of the proximal-type algorithm for nonconvex problems~\cite{AB09,ABS13}. By checking the standard assumption presented in~\cite{AB09,ABS13}, the global convergence and the local convergence rate of the sequence generated by Algorithm~\ref{alg:pmm} can be obtained under the KL assumption on \(\varphi\). 

\item[(ii)] For any \(\epsilon > 0\), denote \(\overline{K} := \lceil \frac{2\bar{\eta}^2}{\mu_1 - \mu_2}(\varphi(x_0) - \varphi_*)\epsilon^{-2 - \delta}\rceil\). Some iterate \(x_k\), \(k = 0, \cdots, \overline{K} + 1\) generated by Algorithm~\ref{alg:pmm} will satisfy \(\|\mathcal{G}(x_k)\|\leq \epsilon\). Suppose that \(\|\cG(x_{k+1})\| > \epsilon\) for all \(k = 0, 1, \cdots, \overline{K}\). From~\eqref{eq:dphi}, we have 
\begin{align*}
\varphi(x_0) - \varphi(x_{\overline{K}+1}) =& \sum_{l=0}^{\overline{K}}\varphi(x_l) - \varphi(x_{l+1}) \\
 \geq&   \frac{\mu_1 -\mu_2}{2}\sum_{l=0}^{\overline{K}}\min\{1, \|\mathcal{G}(x_k)\|^{\delta}\}\|x_{k+1}  - x_k\|^2\\
\geq &\frac{\mu_1 -\mu_2}{2\bar{\eta}^2}\sum_{l=0}^{\overline{K}}\min\{1, \epsilon^{\delta}\}\epsilon^2 \geq \frac{\mu_1 -\mu_2}{2\bar{\eta}^2}(\overline{K}+1)\epsilon^{2+\delta}\\
 >& \varphi(x_0) - \varphi_*,
\end{align*}
where the second inequality follows from Theorem~\ref{th:limitspwithoutls} (a) and the last inequality follows from the definition of \(\overline{K}\). The above inequality contradicts the definition of {\(\varphi_*\)}, so the statement holds. 

\item[(iii)] When \(L_H\) is unknown, it follows from the proof of Theorem~\ref{th:limitspwithoutls} (b) that the constant \(L_H\) in \(\widehat{H}_k\) is mainly used to ensure 
\begin{align*}
f(x_{k+1}) - f(x_k) - \langle\nabla f(x_k), x_{k+1} - x_k\rangle \leq \frac{L_H}{2}\|x_{k+1} - x_k\|^2.
\end{align*}
Let \(\overline{U}_H^0 > 0\) be an estimate of \(L_H\). Define \(H_k(L_H^k) = B_k + (L_H^k + c + \mu_1\min\{1, \|\mathcal{G}(x_k)\|^{\delta}\})I\) and \(\tilde{f}(x; L_H^k) = f(x) - f(x_k) - \langle\nabla f(x_k), x-x_k\rangle - \frac{L_H^k}{2}\|x - x_k\|^2\). We can consider the following scheme (Algorithm~\ref{alg:pmm2}) instead of Algorithm~\ref{alg:pmm} when the constant \(L_H\) is unknown. The line search strategy based on the Lipschitz constant \(L_H\) has been discussed in reference~\cite{NP06}.

\begin{algorithm*}[th!]
\caption{An adaptive inexact proximal Newton method with line search on \(L_H^k\).}\label{alg:pmm2}
\begin{algorithmic}[1]
\REQUIRE{\(x_0\), \(\delta \in [0, 1]\), \(\mu_1\in(0, 1]\), \(\mu_2\in (0, \mu_1]\) and \(L_H^0 = \overline{U}_H^0\)}
\FOR{\(k = 0, 1, \cdots, \)}
\STATE{Compute}
\[
x(L_H^k) = \argmin_{x}\{q(x; x_k, H_k(L_H^k)) + \ve_k^\top(x - x_k)\}
\]
with \(\|\ve_k\| \leq \frac{1}{4}\min\{1, \|\cG(x_k)\|^{\delta}\}\|x(L_H^k) - x_k\|\). 
\WHILE{\(\tilde{f}(x(L_H^k); L_H^k) > 0\)}
\STATE{\(L_H^k := 2L_H^k\)}
\STATE{Compute}
\[
x(L_H^k) = \argmin_{x}\{q(x; x_k, H_k(L_H^k)) + \ve_k^\top(x - x_k)\}
\]
with \(\|\ve_k\| \leq \frac{\mu_2}{2}\min\{1, \|\cG(x_k)\|^{\delta}\}\|x(L_H^k) - x_k\|\). 
\ENDWHILE
\STATE{set \(x_{k+1} = x(L_H^k)\)}. 
\STATE{set \(L_H^{k+1} = \max\{\frac{1}{2}L_H^k, \overline{U}_H^0\}\)}. 
\ENDFOR
\RETURN{\(\{x_k\}\)}
\end{algorithmic}
\end{algorithm*}

Notice that \(\tilde{f}(x(L_H^k); L_H^k) \leq 0\) for any \(L_H^k \geq L_H\). 
The total amount of additional computations for finding \(L_H^k\) is bounded by \(2(\overline{K}+1) + \log_2(\frac{2L_H}{\overline{U}_H^0})\), where \(\overline{K}\) is given in statement (i).

\item[(iv)] If \(g(x) \equiv 0\) and \(H_k = \nabla^2f(x_k)  + \varsigma_kI\), where \(\varsigma_k = \max\{0, -\lambda_{\min}(\nabla^2f(x_k))\} + L_H + c + \mu_1\min\{1, \|\nabla f(x_k)\|^{\delta}\}\), then Algorithm~\ref{alg:pmm} reduces to the regularized Newton method for nonconvex problem \(\min_x f(x)\) with the update formula \(x_{k+1} = x_k + d_k\), where \(d_k\) satisfies 
\[
\|(\nabla^2f(x_k) + \varsigma_kI)d_k + \nabla f(x_k)\| \leq \frac{\mu_2}{2}\min\{1, \|\nabla f(x_k)\|^{\delta}\}\|d_k\|
\]
is an inexact solution of the regularized Newton equation \((\nabla^2f(x_k) + \varsigma_kI)d = -\nabla f(x_k)\). According to the statement (ii) of this remark, Algorithm~\ref{alg:pmm} returns a point satisfies \(\|\nabla f(x_k)\| \leq \epsilon\) within at most \(\mathcal{O}(\epsilon^{-2-\delta})\) iterations. 
\end{itemize}
\end{remark}

\section{Numerical experiments}\label{sec:numerical}
In this section, we evaluate the effectiveness and efficiency of our proposed method on the \(\ell_1\)-regularized Student's \(t\)-regression and the group penalized Student's \(t\)-regression. All numerical experiments are implemented in MATLAB R2023b running on a computer with an Intel(R) Core(TM) i9-10885U CPU @ 2.40GHz \(\times\) 2.4 and 32GB of RAM.

\subsection{\(\ell_1\)-regularized Student's \(t\)-regression}\label{subsec:st}
 The \(\ell_1\)-regularized Student's \(t\)-regression~\cite{AFHL12} problem takes the form of 
 
 \begin{equation}\label{eq:st}
 \min_x\sum_{i=1}^m\log(1 + (Ax-b)_i^2/\nu) + \lambda \|x\|_1,
 \end{equation}
 where \(\nu > 0\) and \(\lambda > 0\) is the regularized parameter. Problem~\eqref{eq:st} is a special case of Problem~\eqref{eq:ncp} with \(f(x) := \sum_{i=1}^m\log(1 + (Ax-b)_i^2/\nu)\) and \(g(x): = \lambda \|x\|_1\). Problem~\eqref{eq:st} is nonconvex and can be used to find the sparse solution \(\bar{x}\) from \(Ax = b\) when the measure \(b\) with heavy-tailed Student-t error~\cite{AFHL12}. We solve the \(x\)-subproblem in Algorithm~\ref{alg:pnewton} by the semismooth Newton (SSN) method presented in~\cite{LST18}. Details are given in Appendix~\ref{Append: SSN}.

 The test examples are randomly generated in the same way as in~\cite{MU14}. The reference signal \(x^{\rm true}\in\mathbb{R}^n\) of length \(n = 512^2\) with \(k = [n/40]\) nonzero entries, where the \(k\) different indices \(i\in\{1, \cdots, n\}\) of nonzero entries are randomly chosen and the magnitude of each nonzero entry is determined via
 \[
x^{\rm true}_i = \eta_1(i)10^{d\eta_2(i)/20},
 \]
\(\eta_1(i) \in\{-1, +1\}\) is a symmetric random sign and \(\eta_2(i)\) is uniformly distributed in \([0, 1]\). The signal has dynamic range of \(d\) dB.
The matrix \(A\in\mathbb{R}^{m\times n}\) takes \(m = n/8\) random cosine measurements, i.e., \(Ax^{\rm true} = ({\rm dct}(x^{\rm true}))_{J}\), where \(J \subset \{1, \cdots, n\}\) with \(\vert J\vert = m\) is randomly chosen and \({\rm dct}\) denotes the discrete cosine transform.  The measurement \(b\) is obtained by adding Student's t-noise with degree of freedom \(5\) and rescaled by \(0.1\) to \(Ax^{\rm true}\). 

We set \(\lambda = 0.1\|\nabla f(0)\|_{\infty}\) and \(\nu = 0.25\) in Problem~\eqref{eq:st}. We first test parameter \(\delta\) in the definition of \(H_k\) for Algorithm~\ref{alg:pnewton}. We set \(x_0 = A^\top b\), \(\mu_1 = \mu_2 = 0.1\), \(\theta = 0.6\), \(\tau = \min\{10^{-4},\frac{1}{2}c\}\) in Algorithm~\ref{alg:pnewton}. We set \(\gamma = 10^{-4}\), \(\bar{\beta} = 0.5\), \(\hat{\beta} = 0.3\), \(\tilde{\beta} = 0.8\) in the SSN method (Algorithm~\ref{alg:ssn} in Appendix~\ref{Append: SSN}). We stop the SSN method when \(\|V^ts^t + \nabla\phi(z^t)\| \leq \bar{\epsilon}_t\) with \(\bar{\epsilon}_t = \min\{\bar{\epsilon}_{t-1}, \|\nabla\phi(z^t)\|^{(1 + \hat{\beta})}\}\), \(\bar{\epsilon}_0 = 10\), or the iteration number reach to \(50\). We stop Algorithm~\ref{alg:pnewton} when \(\|\mathcal{G}(x_k)\| \leq 10^{-5}\). 
Figure~\ref{fig:deltas} displays the average performance of running time as \(\delta\) changes from \(0\) to \(1\) with step size \(0.05\) over 10 trails for \(d = 20\)dB and \(d = 40\)dB, respectively. We set \(c = 0.05\) and \(0.01\) for \(d = 20\) and \(40\), respectively. It can be seen that larger \(\delta\) is beneficial to improve the speed of the algorithm.  We choose \(\delta = 0.95\) for the subsequent test in this subsection if not explicitly stated. Figure~\ref{fig:cr} displays the convergence rate curves of the iterate sequences in terms of \(\|\mathcal{G}(x_k)\|\) and \(\|x^k - \bar{x}\|\) yielded by Algorithm~\ref{alg:pnewton} with different \(\delta\) for \(d = 80\)dB, where \(\bar{x}\) denotes the output of Algorithm~\ref{alg:pnewton}.  
Figure~\eqref{fig:crng} shows that Algorithm~\ref{alg:pnewton} can achieve better convergence rate in terms of \(\|\mathcal{G}(x_k)\|\) than sublinear when implemented.  Figure~\eqref{fig:crdx} shows that Algorithm~\ref{alg:pnewton} can achieve superlinear convergence rate in terms of \(\|x^k - \bar{x}\|\). 

\begin{figure}[h!]
\begin{minipage}[t]{1\linewidth}
\centering
\includegraphics[width = 0.45\textwidth]{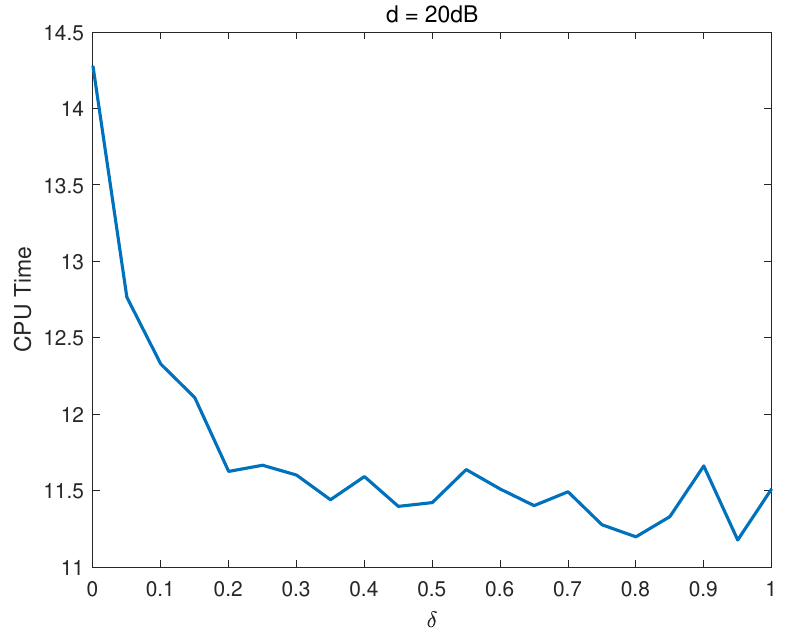}
\includegraphics[width = 0.45\textwidth]{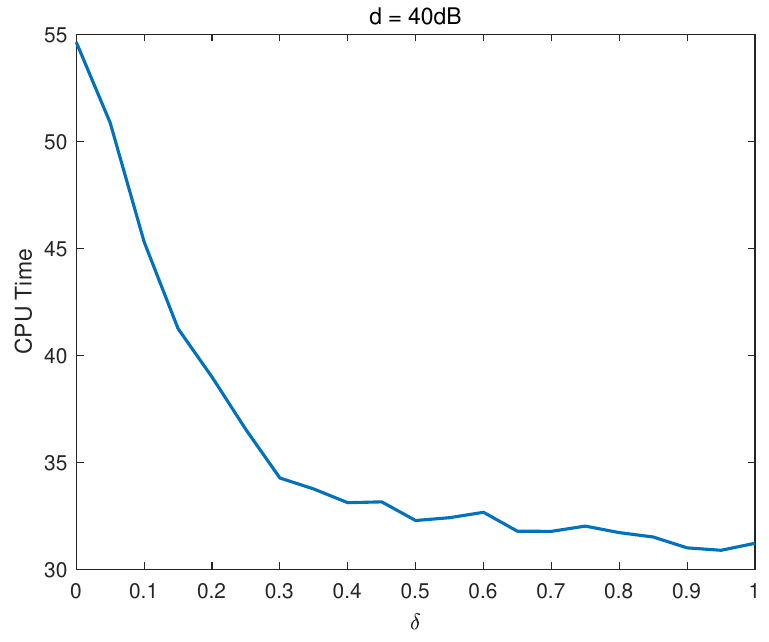}
\end{minipage}
\caption{Average performance of Algorithms~\ref{alg:pnewton} as \(\delta\) changes over \(10\) trails. Left: \(d = 20\)dB. Right: \(d = 40\)dB.}\label{fig:deltas}
\end{figure}
\begin{figure}[h!]
\centering
\subfloat[Convergence of \(\|\mathcal{G}(x_k)\|\)]
{
\label{fig:crng}\includegraphics[width = 0.45\textwidth]{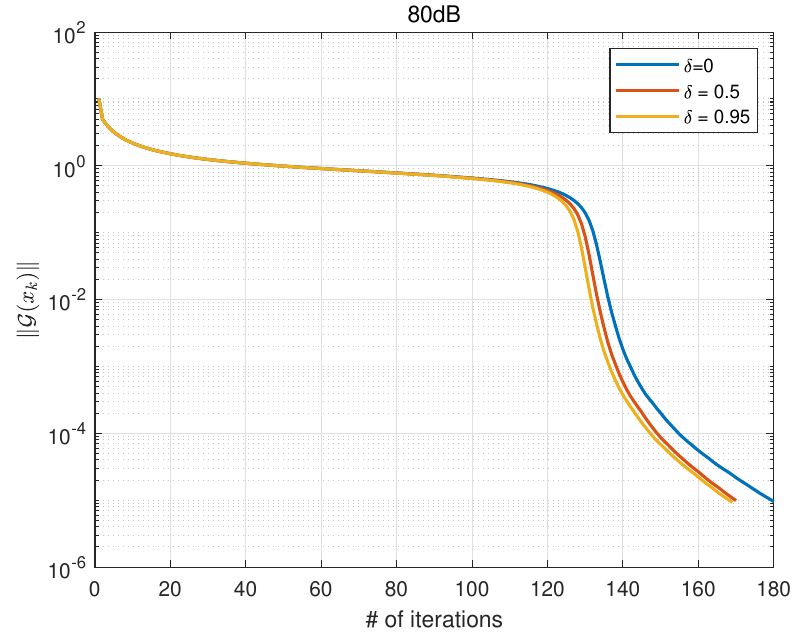}
}
\subfloat[Convergence of \(\|x_k - \bar{x}\|\)]
{
\label{fig:crdx}\includegraphics[width = 0.45\textwidth]{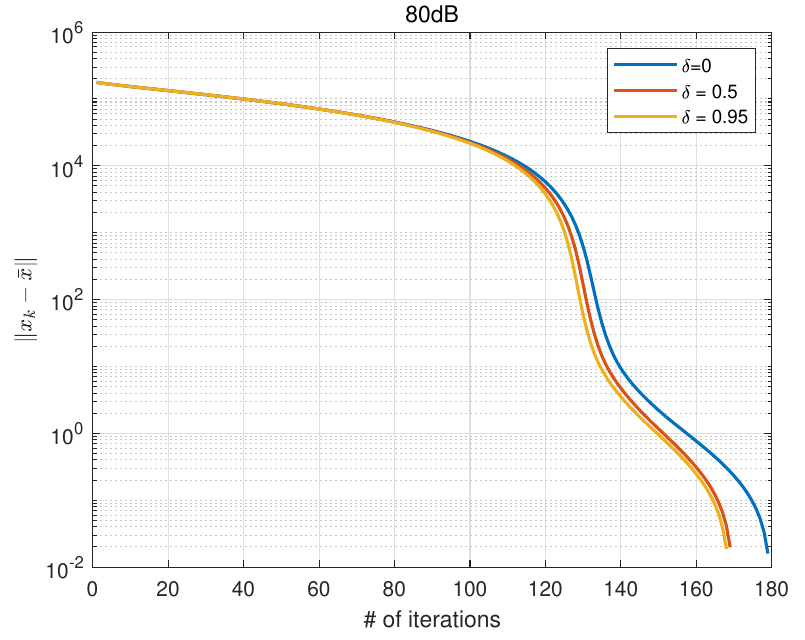}
}
\caption{The convergence behavior of Algorithms~\ref{alg:pnewton} with respect to the number of iterations.}\label{fig:cr}
\end{figure}

Next, we compare Algorithm~\ref{alg:pnewton} with the inexact regularized proximal Newton method (IRPNM for short) proposed in~\cite{LPWY23} and the globalized inexact proximal Newton-type method (GIPN for short) proposed in~\cite{KL21}. The code of IRPNM is collected from \url{https://github.com/SCUP-OptGroup/IRPNM}. All parameters in IRPNM are consistent with the values used in this code. In IRPNM, the \(x\)-subproblems were solved by the semismooth Newton augmented Lagrangian method. In GIPN, the \(x\)-subproblems were solved by using the SSN method proposed in~\cite{MU14} on the KKT system \(r_k(x) = 0\) or performs a proximal gradient step if the return of the SSN method is not satisfactory. In our test, we set \(H_k = \nabla^2f(x_k) + 1.001[-\lambda_{\min}(\nabla^2f(x_k))]_+I\) for GIPN to ensure that the \(x\)-subproblem is convex. We stop each algorithm when \(\|\mathcal{G}(x_k)\| \leq 10^{-5}\), the number of iterations achieves to \(1000\) or the running time achieves to \(1800\)s. Table~\ref{tab:st} lists the average performance of each algorithm over \(10\) trails. Similar performance can be observed for Algorithm~\ref{alg:pnewton} and IRPNM except for \(d = 60\), Algorithm~\ref{alg:pnewton} requires more running time for \(d = 60\). GIPN can not achieve the required accuracy for \(d = 60\) and \(d = 80\). One possible reason is that GIPN may perform proximal gradient steps both in the inner (caused by the SSN method proposed in~\cite{MU14}) and outer loop to obtain \(\hat{x}_k\).  

 \setlength{\tabcolsep}{1pt}
\begin{table}[!th]
\setlength{\abovecaptionskip}{.2cm}
\caption{Numerical comparisons for the \(\ell_1\)-regularized Student's \(t\)-regression.}\label{tab:st}
\scalebox{0.88}{
\centering
\begin{tabular}{c|ccc|ccc|ccc} \hline
\multirow{2}{*}{\(d\) }  & \multicolumn{3}{c|}{Alg. 1} & \multicolumn{3}{c|}{IRPNM} & \multicolumn{3}{c}{GIPN} \\ \cline{2-10}
 & fv        & \(\|\mathcal{G}(x_k)\|\)     & time(s) & fv        & \(\|\mathcal{G}(x_k)\|\)    & time(s) & fv        & \(\|\mathcal{G}(x_k)\|\)     & time(s)\\ \hline
20&9532.54   & 8.17E-06 & {\color{blue}8.59}   & 9532.54   & 8.84E-06 & 12.96  & 9532.54   & 7.41E-06 & 15.08   \\
40& 23812.88  & 6.24E-06 & {\color{blue}24.98}   & 23812.88  & 7.06E-06 & 30.30   & 23812.88  & 5.63E-06 & 123.39  \\
60& 54228.01  & 8.54E-06 & 103.85  & 54228.01  & 7.69E-06 & {\color{blue}79.35}   & 133772.88 & 1.19E+01 & 1801.87 \\
80&134779.26 & 8.21E-06 & {\color{blue}284.67}  & 134779.26 & 8.37E-06 & 328.69  & 490372.40 & 7.63E+00 & 1801.51 \\ \hline
\end{tabular}
}
\end{table}%

\subsection{Group penalized Student's \(t\)-regression}\label{sec:gpt}

Group penalized Student's \(t\)-regression can be formulated by  
\begin{equation}\label{eq:gsr}
\min_{x\in\mathbb{R}^n}\sum_{i=1}^m\log(1 + (Ax-b)_i^2/\nu) + \lambda\sum_{i=1}^l\|x_{J_i}\|, 
\end{equation}
where \(\nu > 0\), \(\lambda > 0\) is the regularized parameter, \(l\in\mathbb{N}\) is the number of groups, and \(\{J_1, \cdots, J_l\}\) is a partition of the index set \(\{1, \cdots, n\}\). Problem~\eqref{eq:gsr} takes the form of  Problem~\eqref{eq:ncp}, where \(f(x) := \sum_{i=1}^m\log(1 + (Ax-b)_i^2/\nu)\) is totally the same as discussed in Section~\ref{subsec:st} and \(g(x) := \lambda\sum_{i=1}^l\|x_{J_i}\|\).  

We generate the reference signal \(x^{\rm true}\), \(A\), and \(b\) in the same way as in Section~\ref{subsec:st}. 
We set \(\lambda = 0.1\|\nabla f(0)\|_{\infty}\) and \(\nu = 0.2\) in Problem~\eqref{eq:gsr} and set \(x_0 = A^\top b\) and \(\delta = 0\) in Algorithm~\ref{alg:pnewton}. 
Observe that the only difference between Problems~\eqref{eq:st} and~\eqref{eq:gsr} is the nonsmooth mapping \(g(x)\). The SSN method discussed in Appendix~\ref{Append: SSN} still works on the \(x\)-subproblem in Algorithm~\ref{alg:pnewton}.  We use the same parameters for the SSN method as in Section~\ref{subsec:st}. 
Table~\ref{tab:gst} reports the average objective values, KKT residuals and running times over \(10\) independent trials for \(\epsilon_0 = 10^{-5}\), \(n = 2^{18}\), \(d = \{60{\rm dB}, 80{\rm dB}\}\), \(l = 2^{10}\), and \(s = \{64, 128\}\), where \(s\) denotes the number of nonzero groups. According to the comparison results obtained in Section~\ref{subsec:st}, we compare with IRPNM in this test. The code of IRPNM is collected from \url{https://github.com/SCUP-OptGroup/IRPNM}. All parameters in IRPNM are consistent with the values used in this code. 
It can be seen from Table~\ref{tab:gst} that Algorithm~\ref{alg:pnewton} yields the same objective values as IRPNM does, but requires much less running time than the latter does.   
\begin{table}[!h]
\setlength{\abovecaptionskip}{.2cm}
\caption{Numerical comparisons for the group penalized Student's \(t\)-regression.} \label{tab:gst}
\centering
\setlength{\tabcolsep}{2mm}
\begin{tabular}{c|c|ccc|ccc}\hline
\multirow{2}{*}{\(d\)}  & \multirow{2}{*}{\(s\)} & \multicolumn{3}{c|}{Alg.1}     & \multicolumn{3}{c}{IRPNM}     \\ \cline{3-8}
                    &                    & fv       & \(\|\mathcal{G}(x_k)\|\)  & time(s) & fv       & \(\|\mathcal{G}(x_k)\|\) & time(s) \\ \hline
\multirow{2}{*}{60} & 64                 & 17852.99 & 7.87E-06 & {\color{blue}24.95}   & 17852.99 & 8.10E-06 & 31.22   \\
                    & 128                & 21670.19 & 7.92E-06 & {\color{blue}22.10}   & 21670.19 & 9.14E-06 & 39.93   \\ \hline
\multirow{2}{*}{80} & 64                 & 52741.59 & 7.89E-06 & {\color{blue}62.64}   & 52741.59 & 9.69E-06 & 296.34  \\
                    & 128                & 63451.74 & 8.07E-06 & {\color{blue}79.13 }  & 63451.74 & 9.24E-06 & 459.31  \\ \hline
\end{tabular}
\end{table}

\section{Conclusions}\label{sec:con}
Inexact proximal Newton methods have been studied in the literature, and none of them provide both global and local convergence rates. One of the main differences between these inexact proximal Newton methods is the accuracy criteria of inexactly solving \(x\)-subproblems. In this paper, we propose an inexact proximal Newton method with only one accuracy condition (compared with two accuracy conditions in most existing methods listed in Table~\ref{tab:ac}). Additionally, both global and local convergence rates are provided under mild assumptions. In our future work, we will discuss the stochastic block-coordinate version of Algorithm~\ref{alg:pnewton}.

\section*{Data Availability Statement}
The data that support the findings of this study are available from the author upon reasonable request.

\begin{appendices}
\section{Semismooth Newton method}\label{Append: SSN}
 In this section, we show how to find \(\hat{x}_k\) for each \(k\in\mathbb{N}\) satisfying~\eqref{eq:errvek} by using the semismooth Newton method. 
 Denote \(\psi(x) = \sum_{i=1}^m\log(1 + x_i^2/\nu)\). We have \(f(x) = \psi(Ax - b)\), \(\nabla f(x) = A^\top\nabla\psi(Ax - b)\), and \(\nabla^2f(x) = A^\top\nabla^2\psi(Ax - b) A\). For each \(x_k\), we can choose \(H_k = A^\top(D_k + \rho_k I)A + \hat{\mu}_k I\), where \(D_k = \nabla^2\psi(Ax_k - b)\), \(\rho_k = 0\) if \(\lambda_{\min}(D_k) > 0\), otherwise, \(\rho_k = \frac{c}{2} - \lambda_{\min}(D_k)\), and \(\hat{\mu}_k = c + \mu_1\min\{1, \|\mathcal{G}(x_k)\|^{\delta}\}\). Then the objective function \(q_k(x) := q(x; x_k, H_k)\) of \(x\)-subproblem in Algorithm~\ref{alg:pnewton} takes the form 
 \begin{align*}
q_k(x) =& f_k + \langle A^\top\nabla\psi(Ax_k - b), x - x_k\rangle + \frac{1}{2}\langle x - x_k, A^\top(D_k + \rho_k I)A(x - x_k)\rangle \\
 &+ \frac{\hat{\mu}_k}{2}\|x - x_k\|^2 + g(x), 
  \end{align*}
where \(f_k = f(x_k)\). Denote \(g_k = (D_k + \rho_kI)^{-1/2}\nabla\psi(Ax_k - b)\), \(A_k = (D_k + \rho_kI)^{1/2}A\) and \(b_k = A_kx_k\). We have
\[
q_k(x) = f_k + \langle g_k, y\rangle + \frac{1}{2}\|y\|^2  + \frac{\hat{\mu}_k}{2}\|x - x_k\|^2 + g(x)\]
with \(y = A_kx - b_k\). 
Therefore, the \(x\)-subproblem can be equivalently written as
\begin{equation}\label{eq:xysub}
\min_{x,y}\{\langle g_k, y\rangle + \frac{1}{2}\|y\|^2  + \frac{\hat{\mu}_k}{2}\|x - x_k\|^2 + g(x), \quad {\rm s.t.}~A_kx - y = b_k\}. 
\end{equation}
The Lagrangian function of Problem~\eqref{eq:xysub} is given by
\[
L(x, y; z) = \langle g_k, y\rangle + \frac{1}{2}\|y\|^2 + \frac{\hat{\mu}_k}{2}\|x - x_k\|^2 + g(x) + \langle z, A_kx - y - b_k\rangle, 
\]
where \(z\in\mathbb{R}^m\) is the Lagrangian multiplier. 
Notice that 
\begin{align*}
-\phi(z) := &\inf_{x,y}\{L(x, y; z)\} \\
=& \inf_x\{g(x) \!+\! \frac{\hat{\mu}_k}{2}\|x \!-\! x_k\|^2 \!+\!  \langle z, A_kx\rangle\} \!+\! \inf_y\{ \frac{1}{2}\|y\|^2 \!+\! \langle g_k, y\rangle \!-\! \langle z, y\rangle\} \!-\! \langle b_k, z\rangle\\
=&\inf_x\{\!g(x) \!+\! \frac{\hat{\mu}_k}{2}\|x \!\!-\!\! (x_k \!\!-\!\! \frac{1}{\mu_k}\!A_k^\top z)\|^2\!\} \!+\! \inf_y\{ \!\frac{1}{2}\|y\|^2 \!\!+\! \langle g_k, \!y\rangle \!-\! \langle z, \!y\rangle\!\} \!-\! \frac{1}{2\hat{\mu}_k}\!\|A_k^\top z\|^2 \\
=& g({\rm Prox}_{g/\hat{\mu}_k}(x_k - \frac{1}{\hat{\mu}_k}A_k^\top z)) -\frac{1}{2}\|z - g_k\|^2 - \frac{1}{2\hat{\mu}_k}\|A_k^\top z\|^2\\
&+ \frac{\hat{\mu}_k}{2}\|{\rm Prox}_{g/\hat{\mu}_k}(x_k - \frac{1}{\hat{\mu}_k}A_k^\top z) - (x_k - \frac{1}{\hat{\mu}_k}A_k^\top z)\|^2,
\end{align*}
where \(x = {\rm Prox}_{g/\hat{\mu}_k}(x_k - \frac{1}{\hat{\mu}_k}A_k^\top z)\) and \(y = z - g_k\). 
Therefore, the dual problem of~\eqref{eq:xysub} is given by \(\min_{z}\phi(z)\). \(z^*\) is the solution of the dual problem if and only if \(\nabla \phi(z^*) = 0\). The exact solution of \(x\)-subproblem can be computed simultaneously by \(x^* =  {\rm Prox}_{g/\hat{\mu}_k}(x_k - \frac{1}{\hat{\mu}_k}A_k^\top z^*)\). 
 \(z^*\) can be obtained via solving the following nonsmooth equation:
 \begin{equation}\label{eq:nonsmootheq}
  \nabla\phi(z) := -A_k{\rm Prox}_{g/\hat{\mu}_k}(x_k - \frac{1}{\hat{\mu}_k}A_k^\top z) + z + b_k  -  g_k = 0. 
 \end{equation}
 Notice that \({\rm Prox}_{g/\mu_k}(\cdot)\) with \(g(x) = \lambda\|x\|_1\) is a Lipschitz continuous piecewise affine function and thus strongly semismooth~\cite{FP03}. {Nonsmooth equation~\eqref{eq:nonsmootheq} can be solved by the semismooth Newton (SSN) method~\cite{QS93,LST18}. Let \(\partial^2\phi(z) = \partial(\nabla \phi)(z)\) be the generalized Hessian of \(\phi\) at \(z\). Define
 \[
 \hat{\partial}^2\phi(z) := \frac{1}{\hat{\mu}_k}A_k\partial{\rm Prox}_{g/\hat{\mu}_k}(x_k - \frac{1}{\hat{\mu}_k}A_k^\top z)A_k^\top + I,
 \]
 where \(\partial{\rm Prox}_{g/\hat{\mu}_k}(x_k - \frac{1}{\hat{\mu}_k}A_k^\top z)\) is the Clarke subdifferential of \({\rm Prox}_{g/\hat{\mu}_k}(\cdot)\) at \( x_k - \frac{1}{\hat{\mu}_k}A_k^\top z\). 
 From \cite[Prop 2.3.3 \& Th. 2.6.6]{C83}, we have
 \[
\partial^2\phi(z)(d) \subseteq \hat{\partial}^2\phi(z)(d), \quad \forall d.
 \]
Let \(V = \frac{1}{\hat{\mu}_k}A_kUA_k^\top + I\) with \(U\in \partial{\rm Prox}_{g/\hat{\mu}_k}(x_k - \frac{1}{\hat{\mu}_k}A_k^\top z)\). Then \(V\in \hat{\partial}^2\phi(z)\) is symmetric positive definite. The iteration of SSN presented in~\cite{LST18} takes the form of \(z^{t+1} = z^t + \alpha_t s^t\), where \(\alpha_t > 0\) is the step size, iteration direction \(s^t\) is an (inexact) solution of the linear system 
\begin{equation}\label{eq:es}
V^ts + \nabla \phi(z^t) = 0
\end{equation}
for some \(V^t\in \hat{\partial}^2\phi(z^t)\). We summarize the SSN method presented in Algorithm~\ref{alg:ssn}. Let \(\{z^t\}\) be the sequence generated by Algorithm~\ref{alg:ssn}. Then \(\{z^t\}\) convergences to \(z^*\)~\cite{LST18}. 
For more details on the SSN  method, see~\cite{QS93,LST18}. }
 
\begin{algorithm*}[th!]
\caption{A semismooth Newton method for solving~\eqref{eq:nonsmootheq} (\(SSN(x_k, A_k, \hat{\mu}_k)\)).}\label{alg:ssn}
\begin{algorithmic}[1]
\REQUIRE{\(\gamma\in(0, 1/2)\), \(\bar{\beta}\in(0, 1)\), \(\hat{\beta}\in(0, 1]\), and \(\tilde{\beta}\in(0, 1)\). Choose \(z^0\). }
\FOR{\(t = 0, 1, \cdots, \)}
\STATE{choose \(U^t \in \partial{\rm Prox}_{g/\hat{\mu}_k}(x_k - \frac{1}{\hat{\mu}_k}A_k^\top z^t)\). Let \(V^t = \frac{1}{\hat{\mu}_k}A_kU^tA_k^\top + I\). Solve the linear system~\eqref{eq:es} by the conjugate gradient (CG) method to find \(s^t\) such that }
\[
\|V^ts^t + \nabla \phi(z^t)\| \leq \min\{\bar{\beta}, \|\nabla \phi(z^t)\|^{1 + \hat{\beta}}\}. 
\]
\STATE{(Line search)~Set \(\alpha_t = \tilde{\beta}^{m_t}\), where \(m_t\) is the first nonnegative integer \(m\) for which}
\[
\phi(z^t + \tilde{\beta}^ms^t) \leq \phi(z^t) + \gamma \tilde{\beta}^m\langle \nabla \phi(z^t), s^t\rangle. 
\]
\STATE{Set \(z^{t + 1} = z^t + \alpha_t s^t\). }
\ENDFOR
\RETURN{\(\{z^t\}\)}
\end{algorithmic}
\end{algorithm*}

Notice that \(\frac{1}{\hat{\mu}_k}g(v) = \frac{\lambda}{\hat{\mu}_k}\|v\|_1\) and \({\rm Prox}_{g/\hat{\mu}_k}(v) = {\rm sign}(v) \circ \max\{\vert v\vert - \frac{\lambda}{\hat{\mu}_k}, 0\}\). For \(v^t \triangleq x_k - \frac{1}{\hat{\mu}_k}A_k^\top z^t\), 
we can always choose \(U^t = {\rm Diag}(u^t)\), where the diagonal elements \(u_i^t\) is given by
 \[
 u_i^t = \left\{
 \begin{array}{ll}
 0 & \quad {\rm if}~\vert v_i^t\vert \leq \frac{\lambda}{\hat{\mu}_k},\\
 1 & \quad {\rm otherwise},
 \end{array}
 \right.\quad i = 1, \cdots, n. 
 \]
Let \(\hat{z}_k := SSN(x_k, A_k, \hat{\mu}_k)\) be the return of Algorithm~\ref{alg:ssn} and \(\hat{x}_k = {\rm Prox}_{g/\hat{\mu}_k}(x_k - \frac{1}{\hat{\mu}_k}A_k^\top \hat{z}_k)\). Then we have \(\hat{\mu}_k(x_k - \frac{1}{\hat{\mu}_k}A_k^\top \hat{z}_k - \hat{x}_k) \in \partial g(\hat{x}_k)\). 
 Notice that the first-order optimality condition of Problem~\eqref{eq:xsub} leads to 
 \[
 -\varepsilon_k \in \nabla f(x_k) + H_k(\hat{x}_k - x_k) + \partial g(\hat{x}_k). 
 \]
It is easy to verify that the accuracy condition~\eqref{eq:errvek} holds when \(\hat{z}_k\) satisfies  
 \begin{equation}\label{eq:xssn}
 \|\!\nabla \!f(x_k) \!+\! \hat{\mu}_k(x_k \!-\! \frac{1}{\hat{\mu}_k}A_k^\top \hat{z}_k \!-\! \hat{x}_k) \!+\! H_k(\hat{x}_k \!-\! x_k)\| 
\!\leq\!  \frac{\mu_2}{2}\min\{1, \|\mathcal{G}(x_k)\|^{\delta}\}\|\hat{x}_k \!-\! x_k\|.
 \end{equation}
 Next, we show there exists \(\hat{z}_k\) such that~\eqref{eq:xssn} holds. Denote \(x_k^t = {\rm prox}_{g/\hat{\mu}_k}(x_k - \frac{1}{\hat{\mu}_k}A_k^\top z^t)\), where \(\{z^t\}_{t\in\mathbb{N}}\) is the sequence generated by Algorithm~\ref{alg:ssn}. From the first-order optimality condition, we have 
 \[
 0 \in \partial g(x_k^t) + \hat{\mu}_k(x_k^t - x_k) + A_k^\top z^t, \quad t\in\mathbb{N}.
 \]
 From~\cite{LST18}, we know that Algorithm~\ref{alg:ssn} satisfies \(\nabla \phi(z^t)\to 0\) as \(t \to\infty\), which yields \(A_k^\top \nabla \phi(z^t) \to 0\) as \(t\to \infty\). Moreover, 
\begin{align*}
-A_k^\top \nabla\phi(z^t) = & \nabla f(x_k) - A_k^\top z^t + A_k^\top A_k(x_k^t - x_k)\\
=& \nabla f(x_k) + H_k(x_k^t - x_k) + \hat{\mu}_k(x_k - x_k^t) - A_k^\top z^t.  
\end{align*}
Suppose \eqref{eq:xssn} does not hold for any \(t \in \mathbb{N}\), that is, 
\[
\|\nabla f(x_k) + H_k(x_k^t - x_k) + \hat{\mu}_k(x_k - x_k^t) - A_k^\top z^t\| > \frac{\mu_2}{2}\min\{1, \|\mathcal{G}(x_k)\|^{\delta}\}\|x_k^t - x_k\|.
\]
We have 
\[
0 \leq \frac{\mu_2}{2}\min\{1, \|\mathcal{G}(x_k)\|^{\delta}\}\|x_k^t - x_k\| \leq \|-A_k^\top \nabla \phi(z^t)\| \to 0, \quad {\rm as}~t\to\infty, 
\]
which implies \(x_k^t \to x_k\) as \(t\to \infty\), that is, \(x_k = \arg\min_x q_k(x)\) and \(\mathcal{G}(x_k) = 0\). This contradicts \(\|\mathcal{G}(x_k)\| > \epsilon_0\). 
Hence, we can choose \(\hat{x}_k := x_k^t\) such that~\eqref{eq:xssn} holds. 
\end{appendices}

\end{document}